\documentclass{amsart}

\newtheorem{thmx}{Theorem}

\usepackage[english]{babel}
\usepackage[utf8]{inputenc}
\usepackage{amsmath}
\usepackage{amsfonts}
\usepackage{mathrsfs}
\usepackage{mathtools}
\mathtoolsset{showonlyrefs}

\usepackage{esint}
\usepackage[colorlinks=true, allcolors=red, linktocpage=true]{hyperref}
\usepackage{amssymb}
\usepackage{graphicx}
\usepackage{commath}
\usepackage{cases}
\usepackage{bm}
\usepackage{dsfont}
\usepackage{xcolor}
\usepackage{verbatim}
\usepackage{todonotes}
\usepackage{enumitem}

\DeclareRobustCommand{\SkipTocEntry}[5]{}

\usepackage{amsthm}

\theoremstyle{plain}

\newtheorem{thm}{Theorem}[section]
\newtheorem{prop}[thm]{Proposition}
\newtheorem{lem}[thm]{Lemma}

\newtheorem{cor}[thm]{Corollary}

\numberwithin{equation}{section}

\theoremstyle{definition}

\newtheorem{defn}[thm]{Definition}

\newtheorem{rem}[thm]{Remark}

\makeatletter
\newcommand*{\rom}[1]{\expandafter\@slowromancap\romannumeral #1@}
\makeatother

\newcommand{\ve}{\varepsilon}

\newcommand{\N}{\mathbb{N}}

\newcommand{\R}{\mathbb{R}}

\newcommand{\Z}{\mathbb{Z}}

\newcommand{\calD}{\mathcal{D}}

\newcommand{\calL}{\mathcal{L}}

\newcommand{\calQ}{\mathcal{Q}}

\newcommand{\cubes}{\calD}

\def\supp{\mathop\mathrm{supp}} 
\def\diam{\mathop\mathrm{diam}}

\def\RCD{\mathrm{RCD}}
\def\myHess{\mathrm{Hess}}
\def\Test{\mathrm{Test}}
\def\HS{\mathrm{HS}}
\def\CPI{C_\mathrm{PI}}
\def\API{A_\mathrm{PI}}
\def\Cdoub{C_\mu}

\DeclarePairedDelimiter\ceil{\lceil}{\rceil}

\def\Xint#1{\mathchoice
	{\XXint\displaystyle\textstyle{#1}}%
	{\XXint\textstyle\scriptstyle{#1}}%
	{\XXint\scriptstyle\scriptscriptstyle{#1}}%
	{\XXint\scriptscriptstyle\scriptscriptstyle{#1}}%
	\!\int}
\def\XXint#1#2#3{{\setbox0=\hbox{$#1{#2#3}{\int}$ }
		\vcenter{\hbox{$#2#3$ }}\kern-.6\wd0}}

\def\dashint{\Xint-}

\newcommand{\har}[1]{H_{#1}}
\newcommand{\harosc}[1]{H_{#1}^{\mathrm{osc}}}

\newcommand{\harrcd}[1]{H_{#1}^{\mathrm{RCD}}}

\newcommand{\harD}{H}
\newcommand{\cRCD}{\mathfrak{c}}
\newcommand{\cradial}{\varrho}

\title[Quantitative harmonic approximations]{Quantitative harmonic approximations and Dorronsoro's Theorem in metric measure spaces}

\author{Matthew Hyde}

\address{Matthew Hyde\\
	Department of Mathematics and Statistics \\
	University of Jyv\"askyl\"a \\
	P.O. Box 35 (MaD) \\
	FI-40014 University of Jyv\"askyl\"a \\
	Finland}
\email{matthew.j.hyde@jyu.fi}

\subjclass[2020]{46E36, 30L99, 28A75}
\keywords{Differentiability, Poincar\'e inequality, RCD spaces.}
\thanks{M.H. is supported by the Research Council of Finland via the project \textit{Quantitative differentiability and rectifiability in metric spaces}, grant no. 363800.}

\begin{document}

		\begin{abstract}
		Suppose $X$ is an $\RCD(K,N)$ space with $K \in \R$ and $N \in (1,\infty)$. We obtain a characterisation of the Newtonian-Sobolev space $N^{1,2}(X)$ in terms of a quantity which measures to what extent a function is locally (across all scales and locations) well-approximated by harmonic functions. A similar characterisation is obtained which further takes into account the local oscillations of the approximating harmonic functions. The first characterisation is new even when $X = \R^n$; the second characterisation is a version of Dorronsoro's Theorem in RCD spaces and gives a new proof of (a special case) of this theorem in Euclidean space. 
	\end{abstract}
	
	\maketitle

\tableofcontents

\section{Introduction}

It is well known that every Lipschitz function $f \colon \R^n \to \R$ is differentiable at $\calL^n$-a.e. point in $\R^n$, where $\calL^n$ denotes the $n$-dimensional Lebesgue measure -- this is Rademacher's Theorem. While this gives strong structural information about $f$ at infinitesimal scales, it says nothing about the structure of $f$ at coarse scales. A quantitative version of Rademacher's Theorem, which does provide structural information at coarse scales, was proven by Dorronsoro \cite{dorronsoro1985characterization}. The following is a special case of Dorronosoro's Theorem. 

\begin{thm}\label{t:dor}
	Let $f \in L^2(\R^n)$. Then $f \in W^{1,2}(\R^n)$ if and only if 
	\begin{align}\label{e:dor}
		\Omega(f) \coloneqq \int_0^\infty \int_{\R^n} \Omega_f(x,r)^2 \, \frac{dxdr}{r} < \infty.
		\end{align}
		Here,
		\begin{align}
			\Omega_{f}(x,r) = \inf_A \left( \dashint_{B(x,r)} \left(\frac{f - A}{r} \right)^2 \, dx \right)^\frac{1}{2}, 
		\end{align}
		where  the infimum is taken over all affine functions $A \colon \R^n \to \R$. In this case, it follows that $\| f \|_{W^{1,2}(X)}^2 \sim \| f \|_{L^2(X)}^2 + \Omega(f)$.
\end{thm}

In words, Dorronsoro's Theorem says that a function is in $W^{1,2}(\R^n)$ if and only if it is \textit{sufficiently well-approximated} by affine functions across all scales and location, this being quantified by condition \eqref{e:dor}. The main result in \cite{dorronsoro1985characterization} is more general than stated above, providing similar characterisations of fractional potential spaces, however, Theorem \ref{t:dor} has proven effective in a variety of applications. 

A key application of Theorem \ref{t:dor} is to the theory of \textit{uniformly rectifiable sets}. These sets are a quantitative analogue of rectifiable sets and were first introduced and developed by David and Semmes \cite{david1991singular,david1993analysis}. For example, Theorem \ref{t:dor} is used in \cite{david1991singular} to show that uniformly rectifiable sets satisfy the so-called \textit{strong geometric lemma} (a quantitative tangent condition) and in \cite{azzam2021poincare} to show that subsets of Euclidean space supporting a Poincar\'e inequality are uniformly rectifiable. Dorronsoro's Theorem is an example of \textit{quantitative differentiation} i.e., a result which provides information regarding how soon a function reaches a certain threshold of regularity. Such results have been used \cite{cheeger2013lower,cheeger2013quantitative} to study partial regularity theory in geometric analysis and nonlinear PDE, in \cite{li2014coarse} to study quantitative non-embeddability of Carnot groups, and in \cite{cheeger2011compression} to study the sparsest cut problem from theoretical computer science. \\

The purpose of this paper is to prove a version of Dorronsoro's Theorem in the class of metric measure spaces satisfying the \textit{Riemann curvature dimension} condition $\RCD(K,N)$ -- a generalisation of the class of Riemannian manifolds with lower bounds on the Ricci curvature. For the relevant background and historical notes on $\RCD$ spaces, we direct the reader to the surveys \cite{ambrosio2018calculus,gigli2023giorgi}. For now, we note that the $\RCD(K,N)$ condition is a Riemannian counterpart to the \textit{curvature dimension} condition CD$(K,N)$ introduced independently in the works \cite{lott2009ricci} and \cite{sturm2006I,sturm2006II}. While the $\RCD(K,N)$ and CD$(K,N)$ conditions are both notions of having Ricci curvature bounded below, the $\RCD(K,N)$ condition differs from the CD$(K,N)$ condition in that it prevents Finsler-type geometry. Classical examples of $\RCD(K,N)$ spaces are Riemannian manifolds with Ricci curvature bounded below, Ricci limit spaces (see \cite{cheeger1997structure}), and Alexandrov spaces with (sectional) curvature bounded below (see \cite{petrunin2011alexandrov,zhang2010ricii}). Such spaces are not necessarily locally Euclidean, can have both non-unique and non-Euclidean tangents, and can admit conical singularities. 

Any space which satisfies the $\RCD(K,N)$ condition with $K \in \R$ and $N \in (1,\infty)$ satisfies a Rademacher-type theorem. Indeed, every such space is locally doubling and, by \cite{rajala2012interpolated,rajala2012local}, admits a local Poincar\'e inequality in the sense of Heinonen and Koskela \cite{heinonen1998quasiconformal}, see \eqref{e:PI}. The seminal work of Cheeger \cite{cheeger1999differentiability} then implies that every Lipschitz function is \textit{infinitesimally generalised linear} at almost every point, meaning, every blow-up of the function at such a point is harmonic (defined via minimisation of a certain Dirichlet energy) with constant energy density, see \cite{cheeger1999differentiability} for the precise definition. 

A quantitative version of the above result has appeared in \cite{cheeger2012quantitative,cheeger2023quantitative}, where it is shown that any Lipschitz function is well-approximated by harmonic functions at \textit{most} locations and scales. Here, \textit{most} is quantified by a certain \textit{weak Carleson condition}. A different quantitative differentiation result for $\RCD(K,N)$ spaces, also in terms of a weak-type Carleson condition, was proved in \cite{hyde2024ricci} as a consequence of its (quantitative) rectifiability properties. These quantitative results differ from Theorem \ref{t:dor} in the sense that they only give control over the measure of the scales and locations which exhibit bad behaviour, they do not provide integrability of the coefficients which measure the approximations.

\addtocontents{toc}{\SkipTocEntry}
\subsection*{Main results} In Theorem \ref{t:RCD} below we characterise the Newtonian-Sobolev space $N^{1,2}(X)$ (see \cite{shan2000newtonian}) in the case that $X$ is an $\RCD(K,N)$ space. We obtain this characterisation, in the spirit of Theorem \ref{t:dor}, by considering a square function which measures the deviation from a certain model class. Of course, affine functions do not make sense in such generality, so we define a new coefficient $\harosc{f}$ which measures local approximations by harmonic functions while also taking into account the oscillation of the gradient of the approximating function. 

More precisely, suppose $(X,d,\mu)$ is an $\RCD(K,N)$ space. For a function $f \in L^2(\mu)$ and a ball $B(x,r) \subset X,$ define
\begin{align}\label{e:h-osc}
	\harosc{f}(x,r)^2 = \inf_h \dashint_{B(x,r)} \left(\frac{f - h}{r} \right)^2 + \left( g_h - \langle g_h  \rangle_{B(x,r)} \right)^2 \, d\mu,
\end{align}
where the infimum is taken over all functions $h$ which are \textit{harmonic} (see \eqref{e:harmonics}) in $B(x,r)$, $g_h$ is the \textit{minimal weak upper gradient} for $h$ in $B(x,r)$ (see \eqref{e:upper-grad}) and $\langle g_h \rangle_{B(x,r)}$ is the average of $g_h$ over $B(x,r)$. As previously mentioned, harmonicity is defined via minimisation of a certain Dirichlet energy; the minimal upper gradient is a metric analogue of $|\nabla h|.$ 

The definition of $\harosc{f}$ is motivated by the following fact -- a function $h \colon \R^n \to \R$ is affine if and only if it is harmonic and satisfies $g_h \equiv c$ i.e., if and only if it is generalised linear (see \cite{cheeger2012quantitative}). It is tempting to seek a version of Theorem \ref{t:dor} which characterises the Sobolev space in terms of approximations by generalised linear functions. However, in such generality, one cannot expect the existence of any non-constant generalised linear function. It turns out that the quantity \eqref{e:h-osc} is a good replacement for \textit{approximating by generalised linear functions} without assuming the existence of such objects, as demonstrated below.

	\begin{thmx}\label{t:RCD}
	Let $(X,d,\mu)$ be an $\RCD(K,N)$ space with $K \in \R$ and $N \in (1,\infty)$, assume that $X$ is bounded in the case that $K < 0$, and let $f \in L^2(X).$ Then, $f \in N^{1,2}(X)$ if and only if 
	\begin{align}\label{e:RCD}
		H^{\emph{osc}}(f) \coloneqq \int_0^{\diam(X)} \int_X \harosc{f}(x,r)^2 \, \frac{d\mu dr}{r} < \infty. 
	\end{align}
	In this case, we have 
	 \begin{align}
	 	\|\nabla f\|_{L^2(X)}^2 \sim \left\| \frac{f - \langle f \rangle_X}{\diam(X)} \right\|_{L^2(X)}^2 + H^\emph{osc}(f),
	 \end{align}
	 where the implicit constants depend only on $N$ in the case that $K \geq 0$, and only on $N,K$ and $\diam(X)$ in the case that $K < 0$. Here, $\nabla$ is the canonical choice of gradient operator on $X$, see Section \ref{s:prelims}. 
\end{thmx}

We actually prove a version of Theorem \ref{t:RCD} with $\harosc{f}$ replaced by another quantity $\harrcd{f}$ defined in terms of the Hessian on $\RCD$ spaces (see Definition \ref{d:H-RCD}). Theorem \ref{t:RCD} is an immediate consequence of this more general statement. We have stated Theorem \ref{t:RCD} in terms of $\harosc{f}$ since this makes sense in more general metric spaces, for example, in spaces supporting a Poincar\'e inequality. When $X = \R^n$, while it is not clear whether $\harosc{f}(x,r)$ is comparable to $\Omega_{f}(x,r)$, it does hold that that $\harrcd{f}(x,r) \sim \Omega_{f}(x,r)$ (Lemma \ref{l:RCD-Rn}). Thus, the proof presented here for Theorem \ref{t:RCD} gives a new proof of Theorem \ref{t:dor} in that setting. \\

It turns out that the oscillation term in \eqref{e:h-osc} is not necessary to characterise $N^{1,2}(X)$ in terms of harmonic approximations. Indeed, we define a further coefficient $\har{f}$ as in \eqref{e:h-osc} except without the oscillatory term (Definition \ref{d:H-no-osc}). Then, we show the following.

	\begin{thmx}\label{t:RCD-2}
	Let $(X,d,\mu)$ be an $\RCD(K,N)$ space with $K \in \R$ and $N \in (1,\infty)$, assume that $X$ is bounded in the case that $K < 0$, and let $f \in L^2(X).$ Then, $f \in N^{1,2}(X)$ if and only if 
	\begin{align}\label{e:RCD-2}
		H(f) \coloneqq \int_0^{\diam(X)} \int_X \har{f}(x,r)^2 \, \frac{dxdr}{r} < \infty. 
	\end{align}
	In this case, we have 
	\begin{align}
		\|\nabla f\|_{L^2(X)}^2 \sim \left\| \frac{f - \langle f \rangle_X}{\diam(X)} \right\|_{L^2(X)}^2  + H(f),
	\end{align}
	where the implicit constants depend only on $N$ in the case that $K \geq 0$, and only on $N,K$ and $\diam(X)$ in the case that $K < 0$. Here, $\nabla$ is the canonical choice of gradient operator on $X$, see Section \ref{s:prelims}.
\end{thmx}

For any $f \in N^{1,2}(X)$, condition \eqref{e:RCD-2} is immediate from Theorem \ref{t:RCD}. Thus, the main point of Theorem \ref{t:RCD-2} is that \eqref{e:RCD-2} is enough to guarantee $f \in N^{1,2}(X).$ Combining both Theorem \ref{t:RCD} and Theorem \ref{t:RCD-2}, one sees that having good harmonic approximations as in \eqref{e:RCD-2} self-improves to having good harmonic approximations with controlled oscillation as in \eqref{e:RCD}. In the Euclidean case, \eqref{e:RCD-2} self-improves to having good affine approximations in the sense of \eqref{e:dor}. The fact that \eqref{e:RCD-2} implies $f \in N^{1,2}(X)$ is new even in the Euclidean case. To summarise, whenever the coefficients make sense, Theorems \ref{t:dor}, \ref{t:RCD} and \ref{t:RCD-2} imply 
\begin{align}
	\Omega(f) \sim H^\text{osc}(f) \sim H(f). 
\end{align}

It is natural here to compare Theorem \ref{t:RCD-2} with \cite[Theorem 1.1]{alabern2012new} (in the case $p=2$). There, the authors prove a Dorronsoro-type estimate for functions $f \colon \R^n \to \R$ in terms of a square function which quantifies the size of $|f(x) - \langle f \rangle_{B(x,r)}|$ across all scales and locations. Heuristically, this is measuring how far $f$ is from a harmonic function by measuring to what extent is satisfies the mean value property. \\

Finally, along with our main results concerning $\RCD$ spaces, we prove a quantitative differentiability result for Sobolev functions on doubling metric measure spaces supporting a Poincar\'e inequality. By \cite{cheeger1999differentiability}, one can assign for each $f \in N^{1,2}(X)$ a Cheeger derivative $Df$. This differential structure comes with a canonical Cheeger energy and a notion of Cheeger harmonic functions with respect to this energy. Thus, one can define a coefficient $\har{f,D}$ which quantifies how far $f$ is from a Cheeger harmonic function. The coefficient $\har{f}$ from above is in fact just $\har{\nabla , f}$ with $\nabla$ the canonical choice of gradient operator on $\RCD$ spaces. In that case, it turns out that the space of Cheeger harmonic functions with respect to $\nabla$ is exactly the space of harmonic functions considered in \cite{bjorn2011nonlinear}. 

Our result in this setting is the following. We have stated a local version which implies the corresponding global estimate (like this stated in Theorem \ref{t:RCD} and Theorem \ref{t:RCD-2}) after taking $R \to \diam(X)$.

\begin{thmx}\label{t:PI}
	Let $(X,d,\mu)$ be a doubling metric measure space which supports a weak $(1,2)$-Poincar\'e inequality, let $D$ be a Cheeger derivative on $X$ and let $f \in N^{1,2}(X)$. Then, for each ball $B = B(z,R)$,  
	\begin{align}\label{e:theorem-PI}
		\hspace{1.5em}\left\| \frac{f - \langle f \rangle_{B}}{R} \right\|_{L^2(B)}^2 + \int_0^{R} \int_{B} \har{f,D}(x,r)^2 \, \frac{d\mu dr}{r} \lesssim \| Df \|_{L^2(AB)}^2,
	\end{align}
	where $A \geq 1$ is a constant multiple of the dilation constant in the weak Poincar\'e inequality and the implicit constant depends only on the doubling constant and the constants appearing in the Poincar\'e inequality. 
\end{thmx}

\addtocontents{toc}{\SkipTocEntry}
\subsection*{Additional remarks and further questions} 
Besides those results mentioned in the previous section, quantitative differentiation has been developed in a number of other non-Euclidean settings, see \cite{azzam2014quantitative, hofmann1995characterization,hytonen2019heat}.  While still Euclidean, we also mention the main result of \cite{azzam2023quantitative}, which is a version of Theorem \ref{t:dor} for functions defined on uniformly rectifiable subsets of $\R^n.$ 

The phrase \textit{Dorronsoro's Theorem} is often associated with one of the inequalities in Theorem \ref{t:RCD}, namely, $H^{\text{osc}}(f) \lesssim \|\nabla f\|^2.$ There are now a number of different proofs of this inequality, each relying in some way on Euclidean structure or complex techniques. For example, the original proof due to Dorronsoro and the much shorter proof due Azzam \cite{azzam2016bi} rely critically on the Fourier transform; the proof due to Hyt\"onen and Naor \cite{hytonen2019heat} relies on heat flow techniques; the starting point to Orponen's proof \cite{orponen2021integral} is a difficult theorem of Jones \cite{jones1990rectifiable}. The proof here relies only on basic properties of harmonic functions and basic inequalities related to them (e.g., Caccioppoli), and is inspired by the methods presented in \cite{cheeger2012quantitative}. As far as the author is aware, the converse inequality is found only in the original paper of Dorronsoro. Here, we adapt some of the heat flow techniques from \cite{hytonen2019heat} to given an alternative proof of this inequality which works in $\RCD$ spaces.

Our main result characterises $N^{1,2}(X)$ only, while Dorronsoro's original paper characterises (among other things) all $N^{1,p}(X)$ with $1 < p < \infty.$ The most obvious question is whether a version of Theorem \ref{t:RCD} can be shown for $p \neq 2.$ Finally, every bounded Ahlfors regular $\RCD(K,N)$ space supports a weak Poincar\'e inequality and is uniformly rectifiable (see \cite{hyde2024ricci} for this last point). It is natural to ask whether Theorem \ref{t:RCD} can be shown for $X$ supporting a weak Poincar\'e inequality or being uniformly rectifiable. 

\addtocontents{toc}{\SkipTocEntry}
\subsection*{Outline} 
In Section \ref{s:prelims}, we give the necessary preliminaries on Sobolev spaces, Poincar\'e inequalities and RCD spaces. The reader familiar with these concepts may wish to skip this section and refer to it only when referenced in later sections. In Section \ref{s:coef} we define more precisely various $H$-coefficients and prove some fundamental results concerning them. In Section \ref{s:PI}, we prove Theorem \ref{t:PI}. This also takes care of the \textit{only if} statement in Theorem \ref{t:RCD-2} (see Lemma \ref{l:C-implies-B}). In Section \ref{s:RCD} we prove the \textit{only if} statement in Theorem \ref{t:RCD}. In Section \ref{s:if-statement}, we prove the \textit{if} statement in Theorem \ref{t:RCD-2}. As noted above, this also takes care of the \textit{if} statement in Theorem \ref{t:RCD}. 

\addtocontents{toc}{\SkipTocEntry}
\subsection*{Acknowledgments} 
We would like the thank Ivan Violo for a number of helpful comments regarding the exposition and for catching several typos.

\section{Preliminaries}\label{s:prelims}
Throughout the paper, $C$ will denote some absolute constant which may change from line to line. If there exists $C \geq 1$ such that $a \leq Cb,$ then we will write $a \lesssim b.$ If $C$ depends on a parameter $t$, we will write $a \lesssim_t b.$ We will write $a \sim b$ if $a \lesssim b$ and $b\lesssim a.$ We define $a \sim_t b$ similarly.  \\

For a metric space $X = (X,d)$, we denote by $B(x,r)$ the open ball \textit{centred} at $x$ with \textit{radius} $r.$ Thus, 
\begin{align}
	B(x,r) = \{y \in X \colon d(x,y) < r\}.
\end{align}
If $B \subseteq X$ is a ball in $X,$ we may write $x_B$ and $r_B$ to denote a point and radius, respectively, such that $B = B(x_B,r_B).$  Note, in general, a ball does not uniquely determine a centre and radius. We denote the diameter of a set $E \subseteq X$ by $\diam(E)$ i.e., 
\begin{align*}
	\text{diam}(E) &= \sup\{d(x,y) : x,y \in E \}. 
\end{align*} 
Let $\mbox{LIP}(X)$ denote the space of Lipschitz functions on $X$. For an open set $\Omega \subset X$, let $\text{LIP}_{\rm{loc}}(\Omega)$ denote the space of locally Lipschitz functions in $\Omega$ and $\rm{LIP}_c(\Omega)$ the space of Lipschitz functions with compact support in $\Omega$.  \\ 

A \textit{metric measure space} is a triple $(X,d,\mu)$, where $(X,d)$ is a separable metric space and $\mu$ is a non-trivial locally finite Borel-regular measure on $X$. A metric measure space is said to be \textit{doubling} if there exists a constant $\Cdoub \geq 1$ such that 
\begin{align}
	0 < \mu(B) \leq \Cdoub\mu(2B) < \infty 
\end{align}
for all ball $B$ in $X$. We call $\Cdoub$ the \textit{doubling constant}.\\

\addtocontents{toc}{\SkipTocEntry}
\subsection*{Christ-David cubes} At various points in the paper, it will be convenient for us to work with a version of ``dyadic cubes'' tailored to $X$. These are the so-called \textit{Christ-David cubes}, which were first introduced by David \cite{david1988morceaux} and generalized in \cite{christ1990b}. 

\begin{thm}\label{cubes}
	There exist constants $c_0, \rho \in (0,1)$ such that the following holds. Suppose $(X,d,\mu)$ is doubling metric measure space and $X_k$ be a sequence of maximal $\rho^k$-separated nets in $X$. There exists a family of open set $\cubes = \bigcup_{k \in \Z} \cubes_k$ in $X$ such that:
	\begin{enumerate}
		\item For each $k \in \Z,$ $\mu(X \setminus \bigcup_{Q \in \cubes_k} Q) = 0.$ 
		\item If $Q_1,Q_2 \in \cubes = \bigcup_{k}\cubes_k$ and $Q_1 \cap Q_2 \not= \emptyset,$ then $Q_1 \subseteq Q_2$ or $Q_2 \subseteq Q_1.$ 
		\item For $Q \in \cubes,$ let $k(Q)$ be the unique integer so that $Q \in \cubes_k$ and set $\ell(Q) = 5\rho^k.$ Then there is $x_Q \in X_k$ such that
		\begin{align*}
			B(x_Q,c_0\ell(Q)) \subseteq Q \subseteq B(x_Q , \ell(Q)) \eqqcolon B_Q. 
		\end{align*}
	\end{enumerate}
\end{thm}

For the rest of the paper, we will fix constants $(c_0,\rho) \in (0,1)$ such that Theorem \ref{cubes} holds. We call a collection of open sets satisfying the conclusion of Theorem \ref{cubes} a \textit{system of Christ-David cubes}. Give a system of Christ-David cubes $\cubes$ and $Q \in \cubes$, we will let $\cubes(Q)$ and $\cubes_k(Q)$ denote the \textit{descendents of $Q$} and the \textit{descents of $Q$ at level $k$}, respectively. That is,  
\begin{align}
	\cubes(Q) = \{R \in \cubes \colon R \subseteq Q\} \quad \mbox{ and } \quad \cubes_k(Q) = \cubes(Q) \cap \cubes_k.
\end{align}

A key fact that we will use throughout the paper is the following. The proof is easy and so we will omit it. 

\begin{lem}\label{l:bounded-overlap}
	Suppose $(X,d,\mu)$ is a doubling metric measure space and $\cubes$ is a system of Christ-David cubes on $X$. For each $A \geq 1$ there is $C = C(A,\Cdoub)$ such that 
	\begin{align}
		\# \{R \in \cubes_k \colon AB_R \cap AB_Q \neq\emptyset \} \leq C
	\end{align}
	for all $Q \in \cubes_k$ and $k \in \Z$. 
\end{lem}

The following result from \cite{hytonen2012systems} generalizes the so-called \textit{1/3-trick} to the metric setting. 

\begin{thm}\label{t:family-cubes}
	Suppose $(X,d,\mu)$ is a doubling metric measure space. There exists a constant $M \geq 1$, depending only on $c_0, \rho$ and the doubling constant, and finitely many systems of Christ-David cubes $\calD^j, \ j =1,\dots,M,$ satisfying the following. For each ball $B \subset X$, there is $j \in \{1,\dots,M\}$ and $Q \in \calD^j$ with $B \subseteq Q$ and $\ell(Q) \lesssim r_B.$ 
\end{thm}

\begin{rem}\label{r:very-small-ball}
	The proof in \cite{hytonen2012systems} actually show that $B \subseteq \tfrac{c_0}{4}B_Q$, see Remark 4.13 there. 
\end{rem}

The following is an easy consequence of Theorem \ref{t:family-cubes}. 

\begin{cor}\label{c:C-to-1-update}
	Suppose $(X,d,\mu)$ is a doubling metric measure space, $\cubes$ is a system of Christ-David cubes on $X$ and $Q_0 \in \cubes_0$. There is $N = N(\Cdoub) \in \N$ and finitely many systems of Christ-David cubes $\{\cubes^j\}_{j=1}^N$ such that, 
	\begin{align}\label{e:C-to-1-update-1}
		&\mbox{for each $Q \in \cubes$ there is $j \in \{1,\dots,N\}$ and $R_Q \in \cubes^j$} \\
		&\mbox{such that $10B_Q \subseteq R_Q$ and $\ell(R_Q) \lesssim \ell(Q).$}
	\end{align}
	For each $j \in \{1,\dots,N\}$ and $R \in \cubes^j,$ we have
	\begin{align}\label{e:C-to-1-update-2}
		\# \{Q \in \cubes \colon R_Q = R\}\lesssim_{\Cdoub} 1. 
	\end{align} 
	Finally, there is an positive integer $k^* \lesssim 1$ and, for each $j \in \{1,\dots,N\}$, a collection of disjoint cubes $\{Q^{j,i}_0\}_{i \in I_j}$ contained in $3B_{Q_0}$ such that, 
	\begin{align}\label{e:C-to-1-update-3}
		&\mbox{for each $Q \in \cubes_k(Q_0)$ with $k \geq k^*$} \\
		&\mbox{there is } j \in \{1,\dots, N\} \mbox{and } i \in I_j\mbox{ such that } R_Q \subseteq Q^{j,i}.
	\end{align}
\end{cor}

\addtocontents{toc}{\SkipTocEntry}
\subsection*{Upper gradients and Newtonian-Sobolev spaces} 
Upper gradients were first introduced in \cite{heinonen1998quasiconformal}. A non-negative Borel function $g$ on a metric space $(X,d)$ is said to be an \textit{upper gradient} of a real-valued function $u$ if the inequality
\begin{align}\label{e:upper-grad}
	|u(\gamma(a)) - u(\gamma(b)) | \leq \int_\gamma g \, ds
\end{align}
holds for every rectifiable curve $\gamma \colon [a,b] \to X$. Here, the integral of $g$ on the right-hand side is computed with respect to the arc length measure along $\gamma$ induced by the metric $d$. 

Following \cite{shan2000newtonian}, we define the Newtonian-Sobolev space $N^{1,p}(X)$ as the collection of all function $f \in L^p(X)$ which admit an upper gradient $g \in L^p(X).$ 
A semi-norm on $N^{1,p}(X)$ is provided by 
\begin{align}
	\| u \|_{N^{1,p}}^p  =  \| u \|_{L^p}^p  + \inf_{g} \| g \|_{L^p}^p, 
\end{align}
where the infimum is taken over all upper gradients $g$ of $u$. By the result in \cite{shan2000newtonian} and an observation of Haj{\l}asz \cite{hajlasz2003sobolev} in the case $p = 1$, there exists a unique minimal $g_u$, called the \textit{minimal weak upper gradient}, which is an upper gradient along $p$-almost every rectifiable curve and for which 
\begin{align}
		\| u \|_{N^{1,p}}^p  =  \| u \|_{L^p}^p  +  \| g_u \|_{L^p}^p.
\end{align}
Here, \textit{$p$-almost every} is defined with respect to the \textit{$p$-modulus} $\text{Mod}_p$ (see \cite{heinonen2012lectures}) and \textit{minimal} means that $g_u \leq g$ $\mu$-a.e. for any other $g$ which is an upper gradient along $p$-almost every rectifiable curve. 

Given a set $E \subseteq X,$ define the Newtonian space with zero boundary values as 
\begin{align}
	N_{0}^{1,p}(E) \coloneqq \{ u \in N^{1,p}(X) \colon u = 0 \ p\mbox{-quasi-everywhere in } X \setminus E\}.
\end{align}
Here, \textit{$p$-quasi-everywhere} is defined with respect to a certain $p$-capacity in $X$, see for instance \cite{shan2000newtonian}. Finally, for an open set $\Omega \subseteq X$, let 
\begin{align}
	N^{1,p}_{\rm loc}(\Omega) \coloneqq \{u \in L^2_{\rm loc}(\Omega) \colon \eta u \in N^{1,2}(X) \mbox{ for all } \eta \in {\rm LIP}_c(\Omega)\}. 
\end{align}

\addtocontents{toc}{\SkipTocEntry}
\subsection*{Poincar\'e inequalities} 
Following \cite{heinonen1998quasiconformal}, we say a metric measure space supports a \textit{weak $(q,p)$-Poincar\'e inequality} if there are constants $\API ,\CPI \geq 1$ such that 
	\begin{align}\label{e:PI}
		\left(\dashint_B |u - \langle u  \rangle_B |^q  \, d\mu\right)^\frac{1}{q}  \leq \CPI r_B \left( \dashint_{\API B} g^p_u \, d\mu \right)^\frac{1}{p}
	\end{align}
	for all balls $B$ in $X$ and all integrable functions $u \colon B \to \R$. Here, and throughout the remainder of the article,
	\[\langle u \rangle_B \coloneqq \dashint_B u \,d\mu \coloneqq \frac{1}{\mu(B)} \int_B u \, d\mu\]
	denotes the average integral of $u$ over $B$. 
	
	By H\"older's inequality, it is easily seen that a weak Poincar\'e inequality is stronger for larger values of $q$ and smaller values of $p$. If the underlying space is doubling, a weak $(1,p)$-Poincar\'e inequality is known to self-improve. The following result, see \cite[Theorem 5.1]{hajlasz1995sobolev}, demonstrates this. 
	
	\begin{thm}\label{t:upgrade-PI}
		Suppose $(X,d,\mu)$ is a doubling metric measure space which supports a weak $(1, p)$-Poincar\'e inequality. Then, $X$ supports a weak $(q,p)$-Poincar\'e inequality for some $q > p.$ In particular, $X$ supports a weak $(p,p)$-Poincar\'e inequality. 
	\end{thm} 
	
	The following result gives us a way to glue Sobolev functions with zero boundary values on disjoint open sets. 
	
\begin{lem}\label{l:stitch-sobolev}
	Suppose $(X,d,\mu)$ is a doubling metric measure space which supports a weak $(1, p)$-Poincar\'e inequality, and $f \in N^{1,2}(X).$ Let $\{U_i\}_{i \in I}$ be a collection of disjoint open sets in $X$ and $\{\phi_i\}_{i \in I}$ be a collection of functions on $X$ such that $\phi_i \in N^{1,2}_0(U_i)$ and $\|g_{\phi_i} \|_{L^{2}(U_i)} \leq C\|  g_f \|_{L^{2}(U_i)}$ for each $i \in I.$ Then, $f' \coloneqq f + \sum_{i \in I} \phi_i \in N^{1,2}(X).$
\end{lem}

\begin{proof}
	Since $\supp(\phi_i) \subseteq U_i$ and $\{U_i\}_{i \in I}$ is a disjoint collection of sets, one has that $f' \in L^2(X)$. Thus, we only need to check that $f'$ has an upper gradient in $L^2(X)$. By \cite[Lemma 6.2.2]{heinonen2015sobolev}, it suffices to check that there exists $g \in L^2(X)$ which is an upper gradient for $f'$ along $p$-almost every rectifiable curve. 
	
	We consider the function
	\begin{align}
		g \coloneqq g_f + \sum_{i \in I} g_{\phi_i}.
	\end{align}
	Since $\phi_i \in N_0^{1,2}(U_i)$, its minimal weak upper gradient $g_{\phi_i}$ is zero $\mu$-a.e. in $X  \setminus U_i$. Then, since the $U_i$ are disjoint, the estimates $\|g_{\phi_i} \|_{L^{2}(U_i)} \leq C\|  g_f \|_{L^{2}(U_i)}$ for $i \in I$ imply $\| g \|_{L^2(X)} \lesssim \| g_f \|_{L^2(X)}$ i.e., $g \in L^2(X)$. Let $\Gamma$ be the collection of rectifiable curves in $X$. Let $\Gamma_f$ be the rectifiable curves along which $g_f$ fails to be an upper gradient for $f$. Define $\Gamma_{\phi_i}$ similarly. Then, let 
	\begin{align}
		\Gamma' = \Gamma \setminus \left( \Gamma_f  \cup \bigcup_{i \in I} \Gamma_i\right).  
	\end{align}
	By countable sub-additivity of the $p$-modulus (\cite{heinonen2012lectures}), $\Gamma' \subseteq \Gamma$ has full $p$-modulus. Moreover, since for each $\gamma \in \Gamma'$ and $i \in I$, $g_f$ is upper gradient for $f$ along $\gamma$ and $g_{\phi_i}$ is upper gradient for $\phi_i$ along $\gamma$, it is easy to check that $g$ is an upper gradient for $f'$ along $\gamma$. Thus, $g$ is an upper gradient for $f'$ along $p$-almost every rectifiable curve. 
\end{proof}

The following Sobolev-type inequality is proved in \cite{bjorn2002boundary} (see Proposition 3.1), based on arguments given in \cite[Lemma 2.8]{kinnunen2001regularity}.

\begin{lem}\label{l:sobolev-poincare}
	Suppose $(X,d,\mu)$ is a doubling metric measure space which supports a weak $(q, p)$-Poincar\'e inequality. There exists $C > 0$, depending only on the doubling and PI constants, such that, if $B$ is a ball in $X$ with $0 < r_B < (1/3) \diam X$ and $u \in N_0^{1,p}(B)$, then 
	\begin{align}
		\left( \dashint_B | u |^q \, d\mu \right)^\frac{1}{q} \leq Cr \left( \dashint_B g_u^p \, d\mu \right)^\frac{1}{p}.
	\end{align}
\end{lem}

\addtocontents{toc}{\SkipTocEntry}
\subsection*{Derivatives and Cheeger harmonic functions} In \cite{cheeger1999differentiability}, Cheeger studied differentiability properties of Lipschitz functions on doubling metric measure spaces which support a weak Poincar\'e inequality. He proves the following remarkable generalisation of Rademacher's Theorem in this setting.

\begin{thm}\label{t:cheeger}
	Suppose $(X,d,\mu)$ is a doubling metric measure space which supports a weak $(1,p)$-Poincar\'e inequality. There exists a countable collection of Borel sets $U_i$ and Lipschitz functions $\varphi_i \colon X \to \R^N$, where $N$ depends only on the doubling and PI constants, satisfying the following. First, 
	\begin{align}
		\mu\left( X \setminus \bigcup_{i=1}^\infty U_i \right) = 0.
	\end{align}
	Second, for each Lipschitz $f \colon X \to \R$ and almost every $x_0 \in U_i$, there exists a unique $Df(x_0) \in \R^N$ (the derivative of $f$ at $x_0$) such that 
	\begin{align}
		\limsup_{X \ni x \to x_0} \frac{|f(x) - f(x_0) - Df(x_0)\cdot (\varphi_i(x) - \varphi_i(x_0)) |}{d(x,x_0)} = 0. 
	\end{align}
\end{thm}

The operator $D$ can be extended to all functions in the Newtonian-Sobolev class $N^{1,p}(X)$ (see \cite[Theorem 4.47]{cheeger1999differentiability}). This extension is linear, satisfies Leibniz rule, and for each $u \in N^{1,p}(X),$ one has 
\begin{align}
	|D u | \sim g_u  \quad \mu\mbox{-a.e.,} 
\end{align} 
where $|\cdot|$  denotes the Euclidean norm on $\R^N$. This motivates the following definition. 
\begin{defn}\label{d:cheeger-derivative}
	An operator $D \colon N^{1,2}(X) \to H$, for some finite-dimensional Hilbert space $H$, is said to be a \textit{Cheeger derivative} if it is linear, satisfies the Leibniz rule and, for each $u \in N^{1,2}(X),$
	\begin{align}\label{e:deriv-comp-upper-grad}
		\|D u \| \sim g_u  \quad \mu\mbox{-a.e.,} 
	\end{align} 
	where $\|\cdot\|$  denotes the norm on $H$ induced by the inner product.
\end{defn}

In general, Cheeger derivatives are not uniquely determined. However, on certain spaces (e.g., Euclidean and $\RCD(K,N)$ spaces), there are canonical choices of such operators. Below we define the class of Cheeger harmonic functions. Keep in mind that the collection of such functions depends on the choice of Cheeger derivative. Let us fix for the remainder of the section a Cheeger derivative $D \colon N^{1,2}(X) \to H$. We will let $\langle \cdot , \cdot \rangle$ denote the inner product on $H$ and $\|\cdot\|$ denote the norm induced by the inner product.

\begin{defn}\label{d:harmonic}
	Let $U \subseteq X$ be open. A function $u \in N_{\rm loc}^{1,2}(U)$ is said to be \textit{Cheeger harmonic} in $U$ if 
	\begin{align}\label{e:min-energy}
		\int_U \|Du\|^2 \, d\mu \leq \int_U \| D(u+\phi)\|^2 \, d\mu \quad \mbox{ for all } \phi \in N^{1,2}_0(U).
	\end{align}
	Equivalently, if
	\begin{align}\label{e:orthogonality}
		\int_U \langle Du , D\phi \rangle \, d\mu = 0 \quad \mbox{ for all } \phi \in N^{1,2}_0(U).
	\end{align}
\end{defn}

\begin{defn}
	Let $U \subseteq X$ be open and $f \in N^{1,2}(U)$. A function $u \in  N^{1,2}(U)$ is said to be a \textit{solution to the Dirichlet problem for $f$ in the weak sense of traces}, if $u$ is Cheeger harmonic in $U$ and $u = f + \phi$ for some $\phi \in N^{1,2}_0(U).$ 
\end{defn}

The existence of such solutions to the Dirichlet problem is a consequence of the following.

\begin{thm}[{\cite[Theorem 7.12]{cheeger1999differentiability}}]\label{t:existence}
	Suppose $(X,d,\mu)$ is a doubling metric measure space which supports a weak $(1,2)$-Poincar\'e inequality. For each open $U \subset X$ and $f \in N^{1,2}(U)$, there exists a solution to the Dirichlet problem for $f$ in the weak sense of traces.
\end{thm}

Cheeger harmonic functions on metric spaces satisfy a Caccioppoli inequality. The Euclidean proof (see, for example, \cite[Theorem 4.1]{giaquinta2012regularity}) works in this setting, but we include it for completeness.  

\begin{lem}\label{l:caccio-replacement}
	Suppose $(X,d,\mu)$ is a doubling metric measure space which supports a $(1,2)$-Poincar\'e inequality, $U \subseteq X$ is open, and $u \in N^{1,2}_{\rm loc}(U)$ is Cheeger harmonic in $U$. Then, for every ball $B$ such that $2B \subseteq U$ and $\lambda \in \R,$
	\begin{align}
		\int_B \|D u \|^2 \, d\mu \lesssim_{\Cdoub,\CPI} \frac{1}{r_B^2} \int_{2B} |u -\lambda|^2 \, d\mu. 
	\end{align} 
\end{lem}

\begin{proof}
	Without loss of generality, we may assume $r_B =1.$ Let $\eta \colon X \to [0,1]$ be a Lipschitz function such that $\eta \equiv 1$ on $B$, $\eta \equiv 0$ on $X\setminus 2B$, and $|D\eta| \lesssim 1.$ Such a function is easy to construct by the McShane-Whitney extension theorem (the absolute value of the derivative of a Lipschitz function is bounded by the Lipschitz constant). Setting $\phi = (u - \lambda)\eta^2$ in \eqref{e:orthogonality} and applying the Leibniz rule gives
	\begin{align}
		\int_U \|Du\|^2 \eta^2 \, d\mu + \int_U 2  (u-\lambda)\eta \langle Du , D\eta \rangle  \, d\mu =0.
	\end{align}
	Using the fact that $\supp(\eta) \subseteq 2B \subseteq U$ and applying H\"older's inequality, we get 
	\begin{align}
		\int_{2B} \|Du\|^2 \eta^2 \, d\mu &\lesssim \int_{2B} |u-\lambda| \eta \|Du\| \|D\eta\| \, d\mu \\
		&\lesssim \left( \int_{2B} \|Du\|^2 \eta^2 \, d\mu \right)^\frac{1}{2}  \left( \int_{2B} \|D\eta\|^2|u-\lambda|^2  \, d\mu \right)^\frac{1}{2} .
	\end{align}
	Thus, since $\eta \equiv 1$ on $B$, 
	\begin{align}
		\int_B \|Du\|^2 \, d\mu \leq \int_{2B}  \|Du\|^2 \eta^2 \, d\mu \lesssim  \int_{2B} \|D\eta\|^2|u-\lambda|^2  \, d\mu \lesssim  \int_{2B} |u-\lambda|^2  \, d\mu,
	\end{align}
	as required. 
\end{proof}

\addtocontents{toc}{\SkipTocEntry}
\subsection*{RCD spaces} 
In this section, we cover the main properties of RCD spaces needed in the proofs of Theorem \ref{t:RCD} and Theorem \ref{t:RCD-2}. As mentioned in the introduction, we direct the reader to the surveys \cite{ambrosio2018calculus,gigli2023giorgi} and references therein for a thorough introduction to the theory (of which we barely scratch the surface).

  The main feature of RCD spaces that enables us to prove Theorem \ref{t:RCD} is the existence of a Hessian (acting on sufficiently smooth functions), developed in \cite{gigli2018nonsmooth}. We utilize the fact that, when scaled correctly, the local $L^2$-norm of $|\myHess(f)|$ is comparable to the local $L^2$-oscillation of $|\nabla f|,$ see Corollary \ref{c:second-order-caccio} and Lemma \ref{l:PI-RCD}. Before getting to this, we give a brief overview of the gradient ($\nabla$), Laplacian ($\Delta$) and Hessian ($\myHess$) on RCD spaces, and their relevant properties. Since we will not work directly with the definition, we will skip some of the details. See \cite{gigli2018nonsmooth} for the full details.

  \subsubsection*{Sobolev spaces} We are interested in how the aforementioned objects act on Sobolev functions. The Sobolev space utilised in \cite{gigli2018nonsmooth}, denoted there by $W^{1,2}(X)$, is the one developed in \cite{ambrosio2014calculus} and is stated in terms of \textit{test plans}. Similarly to the Newtonian spaces, associated to each $f \in W^{1,2}(X)$ is a \textit{minimal weak upper gradient} denoted by $|Df|_w$. It turns out that the space $W^{1,2}(X)$ coincides with the Newtonian space in the following sense (see \cite{erikssonbique2024curvewise}):
  \begin{enumerate}
  	\item  If $f\in N^{1,2}(X),$ then $f \in W^{1,2}(X)$ and $g_f = |Df|_w$ almost everywhere.
  	\item If $f \in W^{1,2}(X)$, then $f$ has a Borel representative $\overline{f} \in N^{1,2}(X)$ with $g_{\overline{f}} =  |Df|_w$ almost everywhere.
  \end{enumerate}
  Thus, when quoting the various necessary results from the RCD literature, we will equivalently state things in terms of $N^{1,2}(X)$.

\subsubsection*{Gradient operator} A defining feature of RCD spaces is the assumption that $N^{1,2}(X)$ is a Hilbert space. This condition, referred to in \cite{gigli2018nonsmooth} as \textit{infinitesimally Hilbertian}, was introduced in \cite{gigli2015differential}. Assuming this, there is a unique pair $(L^0(TX),\nabla)$, where $L^0(TX) $ is an $L^0$-normed module (see \cite{gigli2018nonsmooth}) and $\nabla \colon N^{1,2}(X) \to L^0(TX)$, called the \textit{gradient operator}, is a linear and continuous map such that 
\begin{align}\label{e:comp-upper-grad}
	|\nabla f | = g_f \quad \mbox{almost everywhere}.
\end{align}
The gradient operator satisfies the Leibniz rule, chain rule, and the locality property, meaning,
\begin{align}
	\nabla f = \nabla g \mbox{ almost everywhere in }  \{f=g\}.
\end{align}
Furthermore, the assumption of infinitesimally Hilbertian implies that $L^0(TX)$ is a \textit{Hilbert module} and can be equipped with a canonical inner product $\langle \cdot , \cdot \rangle.$ In particular, the gradient operator $\nabla$ defines a Cheeger derivative in the sense of Definition \ref{d:cheeger-derivative}.

\subsubsection*{Laplace operator and harmonic functions} Let $U \subseteq X$ be open and let $D(\Delta,U)$ denote the space of all functions $f \in N^{1,2}_{\rm loc}(U)$ for which there exists $h \in L^2(U)$ satisfying
\begin{align}\label{e:harmonics}
	\int hg \, d\mu = - \int \langle \nabla f , \nabla g \rangle \, d\mu  \quad \mbox{for all $g \in N^{1,2}_0(U).$}
\end{align}
The function $h$ is called the \textit{Laplacian of $f$} and is denoted by $\Delta f.$ We will set $D(\Delta) \coloneqq D(\Delta,X)$. An element $f \in D(\Delta, U)$ is said to be \textit{harmonic in $U$} if $\Delta f = 0$. Note that this is equivalent to saying that $f$ is Cheeger harmonic in $U$ (Definition \ref{d:harmonic}) with respect to the differential $\nabla$. 

\begin{rem}
	Since $|\nabla u| = g_u$ $\mu$-a.e. (see \cite{ambrosio2013density}), being harmonic in the above sense is the same as being harmonic in the sense considered in the book \cite{bjorn2011nonlinear}. 
	In the case that $X = \R^n,$ it follows from \eqref{e:comp-upper-grad} that there exists a linear isometry $T$ from $L^0(TX)$ to the space of vector fields on $\R^n$ such that $\nabla_{\text{Euc}} f =T(\nabla f)$ for each $f \in N^{1,2}(\R^n)$, see \cite[Remark 2.2.4]{gigli2018nonsmooth}. In particular, the space of harmonic functions described above coincides with the usual space of harmonic functions in $\R^n.$
\end{rem}

\subsubsection*{The global Hessian} The notion of Hessian in RCD spaces was developed in \cite{gigli2018nonsmooth}. The definition is rather involved and, as such, we refer to the aforementioned paper for the full details. We provide a summary below. \\

The space of test functions on $X$ is 
\begin{align}
	\mbox{Test}(X) \coloneqq \{f \in L^\infty(\mu) \cap \mbox{LIP}(X) \cap D(\Delta)  :  \Delta f \in N^{1,2}(X)\}. 
\end{align}
For each $f \in \mbox{Test}(X),$ there is a mapping $\myHess(f) \colon [L^0(TX)]^2 \to L^0(\mu)$ which is bilinear, symmetric and continuous.  This induces a continuous linear map $\myHess(f) \colon L^0(TX) \to L^0(TX)$ via the formula 
\begin{align}
	\langle \myHess(f)(v),w \rangle = \myHess(f)(v,w), \quad w \in L^0(TX). 
\end{align}
It follows from \cite[Equation 3.3.34]{gigli2018nonsmooth} that 
\begin{align}\label{e:hess-in-grad}
	2\myHess(f)(\nabla f) = \nabla | \nabla f |^2 = 2|\nabla f| \nabla |\nabla f | .
\end{align}
Furthermore (see \cite[Section 3.2]{gigli2018nonsmooth}), there is a function $|\cdot|_\HS \colon L^0(TX) \to L^0(\mu)$ with the property that, for $\mu$-a.e. $x \in X$, 
\begin{align}\label{e:norm}
	| \cdot |_\HS(x) \mbox{ defines a norm on } L^0(TX). 
\end{align} 
The function $|\cdot|_\HS$ is called the \textit{Hilbert-Schmidt norm}. It satisfies 
\begin{align}\label{e:HS}
	| \myHess(f)(v) | \leq | \myHess(f)|_\HS |v| , \quad v \in L^0(TX)
\end{align}
Finally, the mapping
\begin{align}
	f \mapsto \myHess(f) \mbox{ is linear}. 
\end{align}

\subsubsection*{The local Hessian} In general, we would like access to a Hessian for functions which are \textit{locally in $\Test(X)$}. One can make this precise as follows. For an open set $\Omega \subseteq X$, let 
\begin{align}
	\text{Test}_{\text{loc}}(\Omega) = \{f \in \text{LIP}_{\text{loc}}(\Omega) \cap D(\Delta,\Omega) \colon \Delta f \in N^{1,2}_{\text{loc}}(\Omega)\}. 
\end{align}
If $\eta \in \text{Test}(X)$ with $\supp(\eta) \subset \subset \Omega,$ and $u \in \text{Test}_{\text{loc}}(\Omega)$ then $\eta u \in \text{Test}(X)$. Thus, for any $f \in \text{Test}_{\text{loc}}(\Omega)$ we can define a function $|\myHess(f)|_{\HS} \in L^2_{\text{loc}}(\Omega)$ as 
\begin{align}\label{e:def-local-Hess}
	|\myHess(f)|_\HS \coloneqq |\myHess(\eta f)|_\HS, \ \mu\mbox{-a.e. in } \Omega'
\end{align}
for every $\eta \in \text{Test}(X)$ with compact support in $\Omega$ such that $\eta =1$ on $\Omega' \subset \subset \Omega.$ This is well-defined due to the locality property of the Hessian.

\begin{lem}\label{l:reverse-triangle-inequality}
	Let $\Omega \subseteq X$. For each $f,g \in \Test_{\rm{loc}}(\Omega)$ we have 
	\begin{align}\label{e:sub-add}
		\left| |\myHess(f)|_\HS - |\myHess(g)|_\HS  \right| \leq | \myHess(f-g) |_\HS \quad \mu\mbox{-a.e. in } \Omega. 
	\end{align}
\end{lem}

\begin{proof}
	Let $\Omega' \subset \subset \Omega$ and fix some $\eta \in \Test(X)$ with compact support in $\Omega$ which satisfies $\eta = 1$ on $\Omega'$. It follows from \eqref{e:norm}, \eqref{e:def-local-Hess}, the reverse triangle inequality and the linearity of the map $h \mapsto \myHess(h)$ that 
	\begin{align}
		\left| |\myHess(f)|_\HS - |\myHess(g)|_\HS  \right|  &= \left| |\myHess(\eta f)|_\HS - |\myHess(\eta g)|_\HS  \right|  \\
		&\leq | \myHess(\eta(f-g)) |_\HS = | \myHess(f-g) |_\HS
	\end{align}
	$\mu$-a.e. in $\Omega'$. Since $\Omega'$ was arbitrary, this completes the proof. 
\end{proof}

\subsubsection*{Main results} We now have everything needed to state the main results of this section. Let us fix a metric measure space $(X,d,\mu)$ which is $\RCD(K,N)$ for some $K \in \R$ and $N \in (1,\infty)$. 

\begin{rem}
	In general, the dependencies in the following results are going to vary depending on whether $K \geq 0$ or $K < 0$. To encode this dichotomy more efficiently, let us define 
	\begin{align}\label{e:cRCD(X)}
		\cRCD(X) = 
		\begin{cases}
			N & \mbox{ if } K \geq 0 ;\\
			(K,N,\diam(X)) & \mbox{ if } K < 0. 
		\end{cases}
	\end{align}
\end{rem}

It follows from the Bishop–Gromov
inequality in ${\rm CD}(K,N)$ \cite[Theorem 2.3]{sturm2006II} that
\begin{align}
	\mbox{$(X,d,\mu)$ is doubling with constant depending only on $\cRCD(X)$. }
\end{align}
It also supports a $(1,2)$-Poincar\'e inequality. 

\begin{thm}[{\cite{rajala2012interpolated,rajala2012local}}]\label{t:RCD-PI}
	The space $(X,d,\mu)$ supports a weak $(1,2)$-Poincar\'e inequality with Poincar\'e constant $\CPI$ depending only $\cRCD(X)$ and dilatation constant given by $\API = 2$. 
\end{thm}

The next result guarantees the existence of bump functions with bounds on the gradient and Laplacian. 

\begin{prop}[{\cite[Lemma 3.1]{mondino2019structure}}]\label{l:cut-off}
	For every $x \in X$, $R > 0$ and $0 < r < R$, there exists a Lipschitz function $\psi \colon X \to \R$ such that 
	\begin{enumerate}
		\item $0 \leq \psi \leq 1$ on $X$, $\psi \equiv 1$ on $B(x,r)$ and $\supp(\psi) \subset B(x,2r)$; 
		\item $r^2|\Delta \psi| + r| \nabla \psi| \lesssim_{\cRCD(X)} 1.$ 
	\end{enumerate}
	
\end{prop}

The space also supports a Bochner-type inequality, from which we will derive a Caccioppoli-type estimate for the Hessian of a harmonic function (Corollary \ref{c:second-order-caccio}). 

\begin{thm}[{\cite[Theorem 3.3.8]{gigli2018nonsmooth}}]\label{t:bochner-type}
	Given $u \in \Test(X)$, it holds that 
	\begin{align}
		\int_X -\frac{1}{2} \langle \nabla |\nabla u|^2 , \nabla \phi \rangle \, d\mu \geq \int_X \left( |\myHess(u)|_{\HS}^2  + \langle \nabla u , \nabla \Delta u \rangle  + K|\nabla u|^2 \right) \phi \, d\mu 
	\end{align}
	for all $\phi \in L^\infty(X) \cap N^{1,2}(X).$ 
\end{thm}

\begin{cor}\label{c:second-order-caccio}
	Let $\Omega \subset X$ be open. For any $h \colon X \to \R$ which is harmonic in $\Omega$ and any ball $B$ such that $5B \subseteq \Omega$ one has 
	\begin{align}
		\int_B | r_B \myHess (h) |_{\HS}^2 \, d\mu \lesssim_{\cRCD(X)} \int_{2B} |\nabla h|^2 \, d\mu.
	\end{align}
\end{cor}

\begin{proof}
	Fix a function $h$ and a ball $B$ as above. By definition, $h \in D(\Delta,\Omega)$, which implies that $h \in \mbox{LIP}(4B)$ by \cite[Theorem 1.1]{jiang2014cheeger}. By Proposition \ref{l:cut-off}, there exists a cut-off function $\theta \in \Test(X)$ such that $\theta \equiv 1$ on $3B$ and $\supp(\theta) \subseteq 4B.$ It is easy to check that $u \coloneqq \theta h \in \Test(X)$ and $u \equiv h$ on $3B$. Applying Proposition \ref{l:cut-off} again, there exists a cut-off function $\phi$ such that $\phi \equiv 1$ on $B$, $\supp(\phi) \subseteq 2B$ and $r_B^2 | \Delta \phi| \lesssim_{\cRCD(X)} 1.$ Note that $\langle \nabla u , \nabla \Delta u \rangle \phi \equiv 0$ on $X$ since $h$ (and hence $u$) is Cheeger harmonic on $2B$ and $\supp(\phi) \subset 2B.$ Applying Theorem \ref{t:bochner-type} and the integration by parts formula \eqref{e:harmonics} for $\phi$, we now have  
	\begin{align}
		\int_B | \myHess(h) |_\HS^2 \, d\mu &= \int_B | \myHess(u) |_\HS^2 \, d\mu \leq \int_X |\myHess(u)|_\HS^2 \phi \, d\mu \\
		&\leq \int_X \frac{1}{2}|\nabla u|^2 \Delta \phi  - K|\nabla u|^2\phi      \, d\mu \\
		& \lesssim_{\cRCD(X)} r_B^{-2} \int_{2B} |\nabla u|^2 \, d\mu = r_B^{-2} \int_{2B} |\nabla h|^2 \, d\mu,
	\end{align}
	which completes the proof. 
\end{proof}

We finish by checking a weak Poincar\'e inequality for $|\nabla f |$ in terms of $\myHess(f)$.

\begin{lem}\label{l:PI-RCD}
	For each ball $B$ in $X$ and each $f \in \Test(X)$ we have 
	\begin{align}
		\int_B | |\nabla f | - \langle |\nabla f | \rangle_B|^2 \, d\mu \lesssim_{\cRCD(X)} \int_{4B} r_B^2|  \myHess(f)|_\HS^2 \, d\mu.
	\end{align}

\end{lem}

\begin{proof}
	It follows from Theorem \ref{t:upgrade-PI} and Theorem \ref{t:RCD-PI} that $(X,d,\mu)$ supports a weak $(2,2)$-Poincar\'e inequality with constant $\CPI'$ depending only on $\cRCD(X)$ and dilation constant equal to $4$. This and \eqref{e:comp-upper-grad} now imply 
	\begin{align}
		\int_B | |\nabla f | - \langle |\nabla f | \rangle_B|^2 \, d\mu \lesssim_{\cRCD(X)} \int_{4B} | \nabla | \nabla f ||^2 \, d\mu .
	\end{align}
	By locality, $| \nabla | \nabla f || = 0$ almost everywhere in $\{|\nabla f| = 0\}$. By equation \eqref{e:HS} and \eqref{e:hess-in-grad}, $| \nabla | \nabla f || \leq |\myHess(f)|_\HS$ almost everywhere in $\{|\nabla f| >0 \}$. In particular, following on from the previous inequality, we have 
	\begin{align} 
		\int_B |  |\nabla f | - \langle |\nabla f | \rangle_B|^2  \, d\mu \lesssim_{\cRCD(X)} \int_{4 B} | \myHess(f)|^2_\HS \, d\mu ,
	\end{align}
	as required. 
\end{proof}

\section{Properties of the $H$-coefficients}\label{s:coef} Here, we define and study the quantities $\har{f,D}$ and $\harrcd{f}$ (recall the definition of $\harosc{f}$ from the introduction). The first quantity gives a way to measure the local distance to the class of Cheeger harmonic functions, the second quantity gives a way to measure the local distance to the class of harmonic functions with controlled Hessian.

\begin{defn}\label{d:H-no-osc} Let $(X,d,\mu)$ be a doubling metric measure space which supports a weak $(1,2)$-Poincar\'e inequality, and let $D$ be a Cheeger derivative on $X$. For each $f \colon X \to \R,$ $x \in X$ and $0 < r < \diam(X),$ define 
	\begin{align}
		\har{f,D}(x,r)^2 \coloneqq \inf_h \har{f,D}(x,r ; h)^2 \coloneqq  \inf_h \dashint_{B(x,r)} \left(\frac{f - h}{r} \right)^2  \, d\mu, 
	\end{align}
	where the infimum is taken over all function $h \colon X \to \R$ which are Cheeger harmonic (with respect to $D$) in $B(x,r)$. If $X$ is an $\RCD(K,N)$ space, we will write $\har{f} \coloneqq \har{\nabla ,f}.$ 
\end{defn}

\begin{defn}\label{d:H-RCD}
	Let $(X,d,\mu)$ be an $\RCD(K,N)$ space. For each $f \colon X \to \R,$ $x \in X$ and $0 < r < \diam(X),$ define 
	\begin{align}
		\harrcd{f}(x,r)^2 \coloneqq \inf_h \harrcd{f}(x,r;h) = \inf_h \dashint_{B(x,r)} \left(\frac{f - h}{r} \right)^2 + r^2|\myHess(h)|_\HS^2 \, d\mu, 
	\end{align}
	where the infimum is taken over all function $h \colon X \to \R$ which are harmonic in $B(x,r)$. 
\end{defn}

\begin{rem}
	In the case that $B = B(x,r)$ we will on occasion write $\har{f,D}(B)$ and $\har{f,D}(B;h)$ to denote the quantities $\har{f,D}(x,r)$ and $\har{f,D}(x,r;h)$. We will adopt a similar convention with the coefficients $\har{f}$ and $\harrcd{f}$. 
\end{rem}

If $X$ is an $\RCD(K,N)$ space for some $K \in \R$ and $N \in (1,\infty)$ (and $\diam(X) < \infty$ in the case that $K < 0$), it is immediate from Lemma \ref{l:PI-RCD} that 
\begin{align}\label{H-estimate}
	\har{f}(x,r) \leq \harosc{f}(x,r) \lesssim_{\cRCD(X)} \harrcd{f}(x,4r)
\end{align}
for each $x \in X$ and $0 < r < \diam(X).$ The following tells us that, in $\R^n,$ the usual Dorronsoro coefficient $\Omega_f$ is comparable to $\harrcd{f}$. 

\begin{lem}\label{l:RCD-Rn}
For every $x \in \R^n,$ $0 < r < \infty$ and $f \colon \R^n \to \R,$ 
\begin{align}
	\Omega_f(x,r) \leq \harrcd{f}(x,r) \lesssim \Omega_f(x,r).
\end{align}
\end{lem}

\begin{proof}
The fact that $\harrcd{f} \leq \Omega$ is immediate from the definitions since any affine function $A$ is harmonic and $|\myHess(A)|_{\HS} = 0$. Here, $|\cdot|_\HS$ is the usual Hilbert-Schmidt norm of a matrix. For the second inequality, fix $x \in \R^n$, $0 < r < \infty$ and $f \colon \R^n \to \R.$ Let $\ve > 0$ and $h \colon \R^n \to \R$ be harmonic in $B(x,r)$ such that 
\begin{align}\label{e:h-less-than}
	\dashint_{B(x,r)} \left( \frac{f - h}{2r} \right)^2 + |r \myHess(h) |^2_{\HS} \, d\mu \leq \harrcd{f}(x,r) + \ve. 
\end{align}
Define an affine function $A \colon \R^n \to \R$ by 
\begin{align}
	A(y) = h(x) + \langle \nabla h \rangle_{B(x,r)} \cdot (y-x) . 
\end{align}
Then,  
\begin{align}
	\Omega_f(x,r)^2 &\leq \dashint_{B(x,r)} \left(\frac{f - A}{2r}  \right)^2 \, dy \\
	&\lesssim \dashint_{B(x,r)} \left( \frac{f-h}{2r} \right)^2  \, dy + \dashint_{B(x,r)} \left( \frac{h - A}{2r}\right)^2 \, dy,
\end{align}
where in the second inequality we used the fact that for $a,b \geq 0$ we have $(a+b)^2 \leq 4(a^2 + b^2).$ Since $A$ and $h$ are both harmonic in $B(x,r)$, the mean value property implies $\langle h - A \rangle_{B(x,r)} = h(x)  - A(x) = 0.$ Thus, after applying the usual Poincar\'e inequality (recalling that $\nabla A = \langle \nabla h \rangle_{B(x,r)})$ we have 
\begin{align}\label{e:omega-less-than}
	\Omega_f(x,r)^2 &\lesssim \dashint_{B(x,r)} \left( \frac{f-h}{2r} \right)^2  \, dy + \dashint_{B(x,r)} | \nabla h - \langle \nabla h \rangle_{B(x,r)} |^2 \, dy.
\end{align}
Applying the Poincar\'e inequality again, we get 
\begin{align}
	\dashint_{B(x,r)} | \nabla h - \langle \nabla h \rangle_{B(x,r)} |^2 \, dy &= \sum_{i=1}^N \dashint_{B(x,r)} | \partial_{x_i} h - \langle \partial_{x_i}h \rangle_{B(x,r)} |^2\, dy \\
	&\lesssim \dashint_{B(x,r)} | r \nabla \partial_{x_i} h |^2\, dy = \sum_{i,j =1}^n \dashint_{B(x,r)} | r \partial_{x_i x_j} h|^2 \, dy \\
	& = \dashint_{B(x,r)} | r \myHess(h) |^2_\HS \, dy. 
\end{align}
This, together with \eqref{e:h-less-than} and \eqref{e:omega-less-than}, implies $\Omega_f(x,r) \lesssim \harrcd{f}(x,r) + \ve.$ Since $\ve$ was arbitrary, this completes the proof. 
\end{proof}

The following lemma gives an upper bound for the $H$-coefficients and a semi-continuity type property.

\begin{lem}\label{l:bounded-and-semicont}
	Suppose $(X,d,\mu)$ is a doubling metric measure space which supports a weak $(1,2)$-Poincar\'e inequality and $D$ is a Cheeger derivative on $X$. Then, for every $f \colon X \to \R$ and every pair of balls $B \subseteq B',$  
	\begin{align}\label{e:upper-bound-H}
		\har{f,D}(B)^2\mu(B) \lesssim_{\CPI} \| Df \|_{L^2(\API B)}^2
	\end{align}
	and
	\begin{align}\label{e:semi-cont-RCD}
		\har{f,D}(B) \lesssim   \left( \frac{\Cdoub^k r_{B'}}{r_{B}} \right) \har{f,D}(B'),
	\end{align}
	where $k = \ceil{\frac{1}{2}\log(\frac{r_{B'}}{r_B})}.$ If $(X,d,\mu)$ is an $\RCD(K,N)$ space, then \eqref{e:upper-bound-H} and \eqref{e:semi-cont-RCD} also hold with the coefficient $\harrcd{f}$ and the derivative $D = \nabla.$ 
\end{lem}

\begin{proof}
	We shall prove the estimates for $\harrcd{f}$ in the case that $(X,d,\mu)$ is an $\RCD$ space. The proofs is almost identical for $\har{f,D}$, and we omit the details. First, viewing $\langle f \rangle_B$ as a function on $X$ (which is harmonic with zero Hessian), the Poincar\'e inequality implies
	\begin{align}
		\harrcd{f}(B)^2 \leq \dashint_B \left(\frac{f-\langle f \rangle_{B}}{r_B} \right)^2 \, d\mu \lesssim_{\CPI} \dashint_{\API B} | \nabla f |^2 \, d\mu. 
	\end{align}
	Equation \eqref{e:upper-bound-H} is immediate from this. For \eqref{e:semi-cont-RCD}, let $\ve > 0$ and let $h \colon X \to \R$ be a harmonic in $B'$ such that 
	\begin{align}
		\harrcd{f}(B';h) \leq \harrcd{f}(B') + \ve. 
	\end{align}
	Then, since $h$ is harmonic in $B$, $\mu(B') \geq \mu(B) $, and $r_B \geq r_{B'}$, we have 
	\begin{align}
		\harrcd{f}(B)^2 &\leq \dashint_B \left( \frac{f-h}{r_B} \right)^2 \, d\mu + \dashint_B | r_B \myHess(h) |_\HS^2 \, d\mu \\
		&\leq \frac{\mu(B') r_{B'}^2}{\mu(B)r_B^2} \dashint_{B'}  \left( \frac{f-h}{r_{B'}} \right)^2 \,d\mu + \dashint_{B'} | r_{B'} \myHess(h) |_\HS^2 \, d\mu \\
		&\leq \frac{\mu(B') r_{B'}^2}{\mu(B)r_B^2} \dashint_{B'}  \left( \frac{f-h}{r_{B'}} \right)^2 + | r_{B'} \myHess(h) |_\HS^2 \, d\mu . 
	\end{align}
	It is easy to check that our choice for $k$ implies $B' \subseteq 2^k B.$ Thus, by the doubling property and our choice of $h,$
	\begin{align}
		\harrcd{f}(B)^2 \leq \left(\frac{\Cdoub^k r_{B'}}{r_{B}}\right)^2\left( \harrcd{f}(B')^2 + \ve\right). 
	\end{align}
	Now, taking $\ve \to 0$ finishes the proof. 
\end{proof}

For $z \in X$ and $0 < R \leq \diam(X)$, define 
\begin{align}
	H_D(f,z,R) \coloneqq \int_0^R \int_{B(z,R)} \har{f,D}(x,r)^2 \, \frac{d\mu dr}{r},
\end{align}
where we define $B(x,\infty) \coloneqq X$. For a subset $\cubes' \subseteq \cubes$ and parameter $\lambda \geq 1$, define
\begin{align}
	H_D(f,\cubes',\lambda) \coloneqq \sum_{\substack{Q \in \cubes'}} \har{f,D}(\lambda B_Q)^2\mu(Q) 
\end{align}
Define $H^{\rm RCD}(f,z,R)$ and $H^{\rm RCD}(f,\cubes',\lambda)$ similarly. \\

The next result gives us a way to translate between the discrete and continuous Carleson conditions. Such a correspondence is now standard for coefficients satisfying a semicontinuity property as in Lemma \ref{l:bounded-and-semicont}. However, we could not find a precise analogue in the literature, so we have included a proof in the appendix. 

\begin{lem}\label{l:continuous-to-discrete}
	Suppose $(X,d,\mu)$ is a doubling metric measure space which supports a weak $(1,2)$-Poincar\'e inequality and $D$ is a Cheeger derivative on $X$. We have the following implications. 
	\begin{enumerate}
		\item For each $z  \in X$ and $0 < R < \diam(X)$, if $k \in \Z$ is such that $5\rho^k \geq R$ and $\lambda \geq 3$, then 
		\begin{align}
			H_D(f,z,R) \lesssim_{\lambda, \Cdoub} \sum_{\substack{Q \in \cubes_{k} \\ Q \cap B(z,R) \neq\emptyset}} H_D(f,\cubes(Q),\lambda).
		\end{align}
		\item For each $k \in \Z$, $Q \in \cubes_k$ and $\lambda \geq 1$, we have 
		\begin{align}
			H_D(f,\cubes(Q),\lambda) \lesssim_{\Cdoub} H_D(f,x_Q,2\lambda \rho^{-1}). 
		\end{align}
		Moreover, in the case that $R = \diam(X) = \infty$, we have 
		\begin{align}
			H_D(f,\cubes,\lambda) \lesssim_{\lambda, \Cdoub} H_D(f,z,\infty). 
		\end{align}
	\end{enumerate}
	If $(X,d,\mu)$ is an $\RCD(K,N)$ space for some $K \in \R$ and $N \in (1,\infty)$, then the same properties hold for the coefficients $H^{\rm RCD}(f,z,R)$ and $H^{\rm RCD}(f,\cdot,\lambda)$. 
\end{lem}

\section{Estimating the $H$-coefficients}\label{s:PI}

In this section, we prove Theorem \ref{t:PI}. We also show how Theorem \ref{t:PI} implies the \textit{only if} statement in Theorem \ref{t:RCD-2}. Throughout, $X = (X,d,\mu)$ denotes a fixed doubling metric measure space supporting a weak $(1,2)$-Poincar\'e inequality, and $D$ denotes a Cheeger derivative on $X$. We also fix a function $f \in N^{1,2}(X).$ For brevity, we will simply denote 
\begin{align}
	\harD \coloneqq \har{f,D}.
\end{align}

It follows directly from the weak $(2,2)$-Poincar\'e inequality (see Theorem \ref{t:RCD-PI} and Theorem \ref{t:upgrade-PI})  that
\begin{align}
	\left\| \frac{f - f_{B(z,R)}}{R} \right\|_{L^2(B(z,R))}^2 \lesssim \| \nabla f \|_{L^2(B(z,\API R)))}^2 
\end{align}
for all $z \in X$ and $0 < R < \diam(X)$, with the correct dependencies. Thus, it only remains to estimate the second term in \eqref{e:theorem-PI}. For this, we would like to discretise the integral appearing in main statement. Suppose for the moment that $\cubes$ is a system of Christ-David cubes on $X$ and $k \in \Z$ is the smallest integer such that $5\rho^k\geq R$. By minimality, $5\rho^k < \rho^{-1} R$. Thus, if $Q \in \cubes_k$ satisfies $Q \cap B(z,R) \neq\emptyset$ and $A \geq 1$, we have 
\begin{align}
	A B_Q \subseteq B(z,10A \rho^k + R) \subseteq 20A \rho^{-1}B(z,R).
\end{align}

By this observation and Lemma \ref{l:continuous-to-discrete} (1), it suffices to prove the proposition below.

\begin{prop}\label{p:main}
	Let $\calD$ be a system of Christ-David cubes on $X$. For each $Q_0 \in \calD$ we have 
	\begin{align}\label{e:tilde-H}
		\sum_{Q \in \cubes(Q_0)} \harD(3B_Q)^2\mu(Q) \lesssim_{\Cdoub,\CPI,\API} \| D f \|_{L^2(3\API B_{Q_0})}^2.
	\end{align}
\end{prop}

We start by showing the following \textit{a priori} weaker result. 

\begin{lem}\label{l:sum-cube}
	Let $\calD$ be a system of Christ-David cubes on $X$. For each $Q_0 \in \calD$ we have 
	\begin{align}
		\sum_{Q \in \cubes(Q_0)} \harD(Q)^2\mu(Q) \lesssim_{\Cdoub,\CPI} \| D f \|_{L^2(Q_0)}^2.
	\end{align}
\end{lem}

\begin{rem}
	The main difference between Proposition \ref{p:main} and Lemma \ref{l:sum-cube} is that we sum the coefficients $H_f(Q)$ instead of $H_f(3B_Q).$ At the end of the section, we will show how to obtain Proposition \ref{p:main} from Lemma \ref{l:sum-cube} with an application of the \textit{one-third trick} (Theorem \ref{t:family-cubes}). 
\end{rem}

Fix a system of Christ-David cubes $\calD$ on $X$ and $Q_0 \in \calD.$ By scaling we may assume $Q_0 \in \calD_0$ (so that $\ell(Q_0) = 5$). For $Q \in \calD,$ let 
\begin{align}\label{e:h^Q}
	h^Q = f + \phi^Q \in N^{1,2}(Q)
\end{align}
be a solution to the Dirichlet problem for $f$ in the weak sense of traces in $Q$, where $\phi^Q \in N^{1,2}_0(Q)$ (recall from Theorem \ref{t:existence} the existence of such a solution). It is clear from Definition \ref{d:H-no-osc} that $\harD(Q) \leq \harD(Q ; h^Q)$ for each $Q \in \calD$. Thus, our main focus will be on proving the following. 
\begin{lem}\label{l:main}
	We have 
	\begin{align}\label{e:main-est}
		\sum_{Q \in \cubes(Q_0)} \harD(Q ; h^Q)^2\mu(Q) \lesssim \| D f \|_{L^2(Q_0)}^2 .
	\end{align}
\end{lem}

We begin the proof of Lemma \ref{l:main} in earnest. Let $\Z_{\geq 0}$ denote the set of non-negative integers. Define, for each integer $k \in \Z_{\geq 0},$ a function $f_k \colon X \to \R$ by 
\begin{align}\label{e:f_k}
	f_k = f +  \sum_{Q \in \cubes_k(Q_0)}  \phi^Q . 
\end{align}
A basic property of minimal weak upper gradients is that $g_{u+v} \leq g_u + g_v$ and $g_{-u} = g_u$ (see \cite{heinonen2015sobolev}). Using this, along with the fact that each $h^Q = f + \phi^Q$ is a solution to the Dirichlet problem for $f$ in $Q$, equations \eqref{e:deriv-comp-upper-grad} and \eqref{e:min-energy} imply 
\begin{align}
	\| g_{\phi^Q} \|_{L^2(Q)} \leq  \| g_{h^Q} \|_{L^2(Q)} + \| g_f\|_{L^2(Q)} \lesssim \| g_f\|_{L^2(Q)}. 
\end{align}
Thus, by applying Lemma \ref{l:stitch-sobolev} with data $f, \ \{Q\}_{Q \in \cubes_{k}}$ and $\{\phi_Q\}_{Q \in \cubes}$, we have 
\[f_k \in N^{1,2}(X).\]
Moreover, since $\phi^Q \in N^{1,2}_0(R)$ whenever $Q \in \cubes_k \cup \cubes_{k+1}$ and $R \in \cubes_{k}$ are such that $Q \subseteq R$, it follows that 
\begin{align}\label{e:zero-boundary}
	f_{k+1}|_R - f_k|_R = \sum_{\substack{Q \in \cubes_k \cup \cubes_{k+1} \\ Q \subseteq R}} \phi^Q \in N^{1,2}_0(R).
\end{align}
Finally, note that the left-hand side of \eqref{e:main-est} can be written as 
\begin{align}\label{e:discrete-cont}
	\sum_{Q \in \cubes(Q_0)} \harD(Q ; h^Q)^2\mu(Q)  = \sum_{k=0}^\infty \left\| \frac{f - f_k}{\cradial_k} \right\|^2_{L^2(Q_0)}.
\end{align}

Recall the role of the constant $\rho$ from Theorem \ref{cubes}. For each integer $k \in \Z,$ let 
\begin{align}\label{e:crad}
	\cradial_k \coloneqq 5\rho^{k}
\end{align}
Then, set
\begin{align}
	K_1 &= \left\{ k \in \Z_{\geq 0} : \left\| \frac{f_{k+1} - f_k}{\cradial_k} \right\|_{L^2(Q_0)}  \leq \frac{1}{4} \left\| \frac{f - f_k}{\cradial_k} \right\|_{L^2(Q_0)} \right\}. 
\end{align}

\begin{lem}\label{l:conseq}
	Suppose that $\{k,k+1,\dots,M-1\} \subseteq K_1$. Then, we have 
	\begin{align}
		\sum_{i=k}^M \left\| \frac{f - f_{i}}{\cradial_{i}} \right\|_{L^2(Q_0)}^2 \lesssim \left\| \frac{f - f_{M}}{\cradial_{M}}\right\|_{L^2(Q_0)}^2. 
	\end{align}
\end{lem}

\begin{proof}
	Without loss of generality, we may assume that $k = 0$. Squaring both sides and using the inequality $\sum_{i=0}^M a_i^2 \leq (\sum_{i=0}^M a_i)^2$ for $a_i \geq 0,$ it suffices to show 
	\begin{align}\label{e:L1-f}
		\sum_{i=0}^M \left\| \frac{f - f_{i}}{\cradial_i} \right\|_{L^2(Q_0)} \lesssim \left\| \frac{f - f_{M}}{\cradial_M}\right\|_{L^2(Q_0)} .
	\end{align}
Using the triangle inequality, swapping the order of summation, summing over a geometric series and using the definition of $K_1,$ we estimate the left-hand side of \eqref{e:L1-f} from above by
\begin{align}
	 &\hspace{-2em}\sum_{i=0}^M \left[ \left( \frac{\cradial_M}{\cradial_i} \right)\left\| \frac{f - f_{M}}{\cradial_M} \right\|_{L^2(Q_0)} + \sum_{j=i}^{M-1} \left( \frac{\cradial_j}{\cradial_i} \right)\left\| \frac{f_{j+1} - f_{j}}{\cradial_j} \right\|_{L^2(Q_0)} \right] \\
	& = \sum_{i=0}^M \left( \frac{\cradial_M}{\cradial_i} \right) \left\| \frac{f - f_{M}}{\cradial_M} \right\|_{L^2(Q_0)} + \sum_{j=0}^{M-1}\sum_{i=0}^j \left( \frac{\cradial_j}{\cradial_i} \right) \left\| \frac{f_{j+1} - f_{j}}{\cradial_j} \right\|_{L^2(Q_0)} \\
	&\leq 2 \left\| \frac{f - f_{M}}{\cradial_M} \right\|_{L^2(Q_0)} + 2\sum_{j=0}^{M-1} \left\| \frac{f_{j+1} - f_{j}}{\cradial_j} \right\|_{L^2(Q_0)} \\
	&\leq 2 \left\| \frac{f - f_{M}}{\cradial_M} \right\|_{L^2(Q_0)} + \frac{1}{2} \sum_{j=0}^{M-1} \left\| \frac{f - f_{j}}{\cradial_j} \right\|_{L^2(Q_0)}.
\end{align}
Rearranging the above inequality gives \eqref{e:L1-f}. 
\end{proof}

\begin{lem}\label{l:M-conseq}
	For each $M \in \Z_{\geq 0},$ we have 
	\begin{align}\label{e:L2-f}
		\sum_{k=0}^{M} \left\| \frac{f - f_{k}}{\cradial_{k}} \right\|_{L^2(Q_0)}^2 \lesssim \left\| \frac{f - f_{M}}{\cradial_M} \right\|_{L^2(Q_0)}^2 + \sum_{k = 0}^{M - 1} \left\| \frac{f_{k+1} - f_{k}}{\cradial_{k}} \right\|_{L^2(Q_0)}^2.
	\end{align}
\end{lem}

\begin{proof}
	Let $\{k_j\}_{j=1}^N$ be an enumeration of the elements in $K_2 \coloneqq \Z_{\geq 0} \setminus K_1$ between $0$ and $M$ (inclusive), ordered such that $k_j < k_{j+1}$ for all $j.$ Then, the left-hand side of \eqref{e:L2-f} can be written as 
	\begin{align}\label{e:split-conseq}
		 \hspace{1.8em} \sum_{k=0}^{k_1} \left\| \frac{f - f_{k}}{\cradial_{k}} \right\|_{L^2(Q_0)}^2+ \sum_{j = 1}^{N-1} \sum_{k = k_{j} +1}^{k_{j+1}} \left\| \frac{f - f_{k}}{\cradial_{k}} \right\|_{L^2(Q_0)}^2+ \sum_{k = k_N + 1}^{M} \left\| \frac{f - f_{k}}{\cradial_{k}} \right\|_{L^2(Q_0)}^2.
	\end{align}
	Observe that 
	\begin{align}
		\sum_{k=0}^{k_1} \left\| \frac{f - f_{k}}{\cradial_{k}} \right\|_{L^2(Q_0)}^2 \lesssim \left\| \frac{f - f_{k_1}}{\cradial_{k_1}} \right\|_{L^2(Q_0)}^2.
	\end{align}
	Indeed, if $0 \in K_2,$ then $k_1 = 0$ and the estimate is trivial. If $0 \in K_1$, then $k \in K_1$ for all $0 \leq k < k_1$ and the above estimate follows immediately from Lemma \ref{l:conseq}. By applying Lemma \ref{l:conseq} directly to the remaining terms in \eqref{e:split-conseq}, and using the definition of $K_2$ (being the complement of $K_1$), we now have 
	\begin{align}
		\sum_{k=0}^{M} \left\| \frac{f - f_{k}}{\cradial_{k}} \right\|_{L^2(Q_0)}^2 &\lesssim \sum_{j =1}^N  \left\| \frac{f - f_{k_j}}{\cradial_{k_j}} \right\|_{L^2(Q_0)}^2 +  \left\| \frac{f - f_{M}}{\cradial_M} \right\|_{L^2(Q_0)}^2 \\
		&\lesssim \sum_{j =1}^N  \left\| \frac{f_{k_{j}+1} - f_{k_j}}{\cradial_{k_j}} \right\|_{L^2(Q_0)}^2 +  \left\| \frac{f - f_{M}}{\cradial_M} \right\|_{L^2(Q_0)}^2.
	\end{align}
	Inequality \eqref{e:L2-f} is immediate from this. 
\end{proof}

Recalling \eqref{e:discrete-cont}, the following lemma finishes the proof of Lemma \ref{l:main} (hence, also Lemma \ref{l:sum-cube}) after taking $M \to \infty.$ 

\begin{lem}
	For each $M \in \Z_{\geq 0},$ we have 
	\begin{align}
		\sum_{k =0}^M \left\| \frac{f-f_k}{\cradial_k} \right\|^2_{L^2(Q_0)} \lesssim \| D f \|_{L^2(Q_0)}^2. 
	\end{align}
\end{lem}

\begin{proof}
	Fix $k \in \Z_{\geq 0}$ for the moment. By \eqref{e:zero-boundary}, we have that $f_{k+1}|_Q - f_k|_Q \in N^{1,2}_0(Q)$ for each $Q \in \cubes_k.$ Thus, by applying Lemma \ref{l:sobolev-poincare} and using the orthogonality condition \eqref{e:orthogonality} (since $f_k$ is Cheeger harmonic on each $Q \in \cubes_k$), we have 
	\begin{align}
		\begin{split}\label{e:D-orth}
		\left\| \frac{f_{k+1} - f_k}{\cradial_k} \right\|_{L^2(Q)}^2 &\lesssim \| D f_{k+1} - D f_k \|_{L^2(Q)}^2  \\
		&=  \| D f_{k+1} \|_{L^2(Q)}^2 - 2\int_Q \langle Df_{k+1} , D f_{k} \rangle \, d\mu  + \| D f_k \|_{L^2(Q)}^2 \\
		&= \| D f_{k+1} \|_{L^2(Q)}^2 - \| D f_k \|_{L^2(Q)}^2.
	\end{split}
	\end{align}
	Summing over $Q \in \calD_{k}$ then gives 
	\begin{align}
		\left\| \frac{f_{k+1} - f_k}{\cradial_k} \right\|_{L^2(Q_0)}^2 = \| D f_{k+1} \|_{L^2(Q_0)}^2  -\| D f_k \|_{L^2(Q_0)}^2. 
	\end{align}
	Similarly,
	\begin{align}
		\left\| \frac{f - f_M}{\cradial_k} \right\|_{L^2(Q_0)}^2 \lesssim \ \| D f \|_{L^2(Q_0)}^2 - \| D f_M \|_{L^2(Q_0)}^2.
	\end{align}
	Combining Lemma \ref{l:M-conseq} with the previous two estimates and summing over a telescoping series, we get   
	\begin{align}
		\sum_{k =0}^M \left\| \frac{f-f_k}{\cradial_k} \right\|^2_{L^2(Q_0)} \lesssim \| D f \|_{L^2(Q_0)}^2 - \|D f_0\|_{L^2(Q_0)}^2 \leq \| D f \|_{L^2(Q_0)}^2.
	\end{align}
\end{proof}

\begin{proof}[Proof of Proposition \ref{p:main}]
	Let $\calD$ be a system of Christ-David cubes on $X$ and $Q_0 \in \calD.$ Assume without loss of generality that $Q_0 \in \cubes_0$. Let 
	\begin{align}
		N, k^* \in \N, \quad \{\cubes^j\}_{j=1}^N, \quad  \{R_Q\}_{Q \in \cubes} \quad \mbox{and} \quad \{Q_0^{j,i}\}_{i \in I_j}  
	\end{align}
	satisfy the conclusions of Corollary \ref{c:C-to-1-update}. It follows from \eqref{e:C-to-1-update-1} and Lemma \ref{l:bounded-and-semicont} that 
	 \begin{align}\label{e:1}
	 	\harD(3B_Q) \lesssim_{\Cdoub} \harD(R_Q). 
	 \end{align}
	
	By \eqref{e:upper-bound-H} and the fact that the balls $\{3\API B_Q \colon Q \in \cubes_k\}$ have bounded overlap (with constant depending on $\API$ and $\Cdoub$) and are contained in $3\API B_{Q_0}$, we have
	\begin{align}
		\sum_{k=0}^{k^*} \sum_{Q \in \cubes_k(Q_0)}\harD(3B_Q)^2\mu(Q) \lesssim_{\API, \Cdoub} \sum_{k=0}^{k^*} \| Df \|_{L^2(3\API B_{Q_0})}^2 \lesssim  \| Df \|_{L^2(3\API B_{Q_0})}^2.
	\end{align}
	To deal with the sum over smaller cubes, we apply \eqref{e:C-to-1-update-2}, \eqref{e:C-to-1-update-3}, \eqref{e:1} and Lemma \ref{l:sum-cube} to get 
	\begin{align}
		\sum_{k=k^*}^\infty \sum_{Q \in \cubes_k(Q_0)} \harD(3B_Q)^2\mu(Q) &\lesssim_{\Cdoub} \sum_{k=k^*}^\infty \sum_{Q \in \cubes_k(Q_0)} \harD(R_Q)^2\mu(Q)   \\ &\lesssim_{\Cdoub} \sum_{j=1}^N \sum_{i \in I_j} \sum_{R \in \cubes^j(Q_0^{j,i})} \harD(R)^2 \mu(Q) \\
		&\lesssim_{\Cdoub,\CPI} \sum_{j=1}^N \sum_{i \in I_j} \| Df \|_{L^2(Q_0^{j,i})}^2 \lesssim_{\Cdoub} \|Df \|_{L^2(3B_{Q_0})}^2,
	\end{align}
	which finishes the proof. 
\end{proof}

\begin{lem}\label{l:C-implies-B}
	Theorem \ref{t:PI} implies the only if statement in Theorem \ref{t:RCD-2}. 
\end{lem}

\begin{proof}
	Let $(X,d,\mu)$ be an $\RCD(K,N)$ space satisfying the hypotheses of Theorem \ref{t:RCD} and let $f \in N^{1,2}(X)$. Recalling from Definition \ref{d:H-no-osc} and Section \ref{s:prelims} that $\har{f} \coloneqq \har{\nabla,f}$ and $\nabla$ is a Cheeger derivative, it suffices to know that $(X,d,\mu)$ is doubling and supports a weak Poincar\'e inequality with the correct dependencies. This follows from Theorem \ref{t:RCD-PI}.
\end{proof}

\section{Estimating the oscillatory $H$-coefficients in $\RCD(K,N)$ spaces}\label{s:RCD}

In this section, we prove the \textit{only if} directions of Theorem \ref{t:RCD} and Theorem \ref{t:RCD-2}. As we argued at the beginning of Section \ref{s:PI}, it is enough to show that that the integral of the $H$-coefficients can be controlled by the derivative. By appealing to \eqref{H-estimate} and taking $R \to \diam(X)$, this immediate from the following results.

\begin{thm}\label{t:RCD3}
	Let $(X,d,\mu)$ be an $\RCD(K,N)$ space with $K \in \R$ and $N \in (1,\infty)$, and assume that $X$ is bounded in case that $K < 0$. Let $f \in N^{1,2}(X)$. Then, for each $z \in X$ and $0 < R < \diam(X),$ we have 
	\begin{align}
		\int_0^R \int_{B(z,R)} \harrcd{f}(x,r)^2 \, \frac{dxdr}{r} \lesssim_{\cRCD(X)} \| \nabla f \|_{L^2(B(z,AR))}^2,
	\end{align}
	with $A = 120\rho^{-1} \API$ and $\cRCD(X)$ as in \eqref{e:cRCD(X)}.
\end{thm}

Proving Theorem \ref{t:RCD3} is our goal for the remainder of the section. Fix $(X,d,\mu)$ satisfying the hypotheses of Theorem \ref{t:RCD3} and fix a Sobolev function $f \in N^{1,2}(X)$ for the remainder of the section. By a similar argument to that at the beginning of Section \ref{s:PI}, it suffices to show the following.
\begin{prop}\label{p:main-second-order}
	Suppose $\cubes$ is a system of Christ-David cubes on $X$ and $Q_0 \in \cubes$. Then,
	\begin{align}
		\sum_{Q \in \cubes(Q_0)} \harrcd{f}(3B_Q)^2\mu(Q) \lesssim_{\cRCD(X)} \| \nabla f \|_{L^2(6\API  B_{Q_0})}^2.
	\end{align}
\end{prop}

As in the previous section, we will show the \textit{a priori} weaker estimate stated in Lemma \ref{l:main-second-order} below. Modulo replacing $\har{f}$ with $\harrcd{f},$ the fact that Lemma \ref{l:main-second-order} is sufficient for Proposition \ref{p:main-second-order} can be observed by an identical argument to the fact that Lemma \ref{l:main} is sufficient for Proposition \ref{p:main}, see the end of Section \ref{s:PI}. We omit the details. 

\begin{lem}\label{l:main-second-order}
	Suppose $\cubes$ is a system of Christ-David cubes on $X$ and $Q_0 \in \cubes$. Then, 
	\begin{align}
		\sum_{Q \in \cubes(Q_0)} \harrcd{f}(Q)^2\mu(Q) \lesssim_{\cRCD(X)} \| \nabla f \|_{L^2(3\API B_{Q_0})}^2. 
	\end{align}
\end{lem}

\begin{rem}
	Recalling from Theorem \ref{t:RCD-PI} that the doubling and Poincar\'e constant $(\Cdoub,\CPI)$ depend only on $\cRCD(X)$, we can allow our estimates below to depend on these quantities. Since $\API = 2$, we will not keep track of how the implicit constants depend on $\API$.
\end{rem}

We begin the proof of Lemma \ref{l:main-second-order} in earnest. Fix a system of Christ-David cubes $\cubes$ on $X$ and $Q_0 \in \cubes$. Without loss of generality, assume that $Q_0 \in \cubes_0$ so that $\ell(Q_0) = 5.$ Let 
\begin{align}
	N, k^* \in \N, \quad \{\cubes^j\}_{j=1}^N, \quad  \{R_Q\}_{Q \in \cubes} \quad \mbox{and} \quad \{Q_0^{j,i}\}_{i \in I_j}  
\end{align}
satisfy the conclusions of Corollary \ref{c:C-to-1-update}. The following will be used several times.

\begin{lem}\label{l:control-R_Q}
	We have 
	\begin{align}\label{e:control-R_Q}
		\sum_{k=k^*}^\infty \sum_{Q \in \cubes_k(Q_0)} \har{f}(R_Q)^2 \mu(R_Q) \lesssim_{\Cdoub, \CPI} \| \nabla f \|_{L^2(3B_{Q_0})}^2 .
	\end{align}
\end{lem}

\begin{proof}
	By applying \eqref{e:C-to-1-update-2} and \eqref{e:C-to-1-update-3}, we have that 
	\begin{align}\label{e:above'}
		\sum_{k=k^*}^\infty \sum_{Q \in \cubes_k(Q_0)} \har{f}(R_Q)^2 \mu(R_Q) \lesssim_{\Cdoub} &\sum_{j=1}^N \sum_{i \in I_j} \sum_{Q \in \cubes^j(Q_0^{j,i})} \har{f}(Q)^2\mu(Q). 
	\end{align}
	Then, by applying Lemma \ref{l:sum-cube}, using that $N = N(\Cdoub)$ and that $\{Q_0^{j,i}\}$ are disjoint cubes contained in $3B_{Q_0}$, the right-hand side of \eqref{e:above'} is bounded by a constant multiple of $\|\nabla f \|_{L^2(3B_{Q_0})}^2$, with the constant depending only on $\cRCD(X)$, $\Cdoub$ and $\CPI$. 
\end{proof}

For each $Q \in \cubes_{k^*-1}$, let $l_{Q} \colon X \to \R$ be the constant function 
\begin{align}
	l_Q(x) = \langle f \rangle_{3B_Q}. 
\end{align}
For each $Q \in \cubes_k$ with $k \geq k^*,$ let $l_{Q} \colon X \to \R$ be a function which is harmonic in $R_Q$ and satisfies 
\begin{align}\label{e:optimal-l}
	\har{f}(R_Q ; l_{Q}) \lesssim \har{f}(R_Q).
\end{align}
For each integer $k \geq k^*-1$, define $g_k \colon Q_0 \to \R$ by 
\begin{align}\label{e:def-g_k}
	g_k = \sum_{Q \in \cubes_k(Q_0)} | \myHess(l_Q) |_\HS \mathds{1}_Q. 
\end{align}
Note that $g_{k^*-1} \equiv 0$. 
\begin{lem}\label{l:upper-osc}
	We have 
	\begin{align}\label{e:upper-osc}
		\sum_{Q \in \cubes(Q_0)} \harrcd{f}(Q)^2\mu(Q) \lesssim_{\Cdoub,\CPI} \| \nabla f \|_{L^2(\lambda B_{Q_0})}^2 + \sum_{k=k^*}^\infty \| \cradial_k g_k \|_{L^2(Q_0)}^2,
	\end{align}
	where $\lambda = \max\{3,\API\}$. 
\end{lem}
\begin{proof}
	By Lemma \ref{l:bounded-and-semicont}, we have 
	\begin{align}
		\harrcd{f}(Q)^2\mu(Q) \lesssim_{\Cdoub,\CPI} \| \nabla f \|_{L^2(\API B_Q)}^2
	\end{align}
	for each $Q \in \cubes.$ Since the collection of balls $\{\API B_Q\}_{Q \in \cubes_k}$ are contained in $\API B_{Q_0}$ and have bounded overlap for each $k \geq 0$ (depending on the doubling constant $\Cdoub$ and $\API$), and $k^* \lesssim 1$, we have 
	\begin{align}
		\sum_{k=0}^{k^*} \sum_{Q \in \cubes_k(Q_0)} \harrcd{f}(Q)^2\mu(Q) \lesssim_{\Cdoub, \CPI,\API} \| \nabla f \|_{L^2(\API B_{Q_0})}^2.  
	\end{align}
	Unravelling the definitions, the sum over cubes in $\cubes_k(Q_0)$ with $k \geq k^*$ is bounded from above by 
	\begin{align}
		&\sum_{k=k^*}^\infty \sum_{Q \in \cubes_k(Q_0)} \harrcd{f}(Q;l_Q)^2\mu(Q) \\
		&\hspace{0.5em} = \sum_{k=k^*}^\infty  \sum_{Q \in \cubes_k(Q_0)} \har{f}(Q ; l_Q)^2\mu(Q) + \sum_{k=k^*}^\infty \sum_{Q \in \cubes_k(Q_0)} \| \diam(Q) |\myHess(l_Q)|_\HS \|_{L^2(Q)}^2 \\
		&\hspace{0.5em} \lesssim \sum_{k=k^*}^\infty  \sum_{Q \in \cubes_k(Q_0)} \har{f}(R_Q ; l_Q)^2\mu(R^Q) + \sum_{k=k^*}^\infty \| \cradial_k g_k \|_{L^2(Q_0)}^2 \\
		&\hspace{0.5em} \lesssim_{\Cdoub,\CPI}  \| \nabla f \|_{L^2(3B_{Q_0})}^2 +  \sum_{k=k^*}^\infty\| \cradial_k g_k \|_{L^2(Q_0)}^2,
		\end{align}
		where in the first inequality we have used the definition of $g_k$ from \eqref{e:def-g_k} and the fact that $Q \subseteq R_Q$, and in the second inequality we used \eqref{e:optimal-l} along with Lemma \ref{l:control-R_Q}.
\end{proof}

By Lemma \ref{l:upper-osc}, the proof of Lemma \ref{l:main-second-order} will be complete if we can bound the final term in \eqref{e:upper-osc} by $\| \nabla f \|_{L^2(3\API B_{Q_0})}^2$. This will be our goal for the remainder of the section. To achieve this, we first partition $\Z_{\geq k^*}$ (the set of integers $k \geq k^*$) into 
\begin{align}
	K_3 &\coloneqq \left\{ k \in \Z_{\geq k^*} \colon \| g_k - g_{k-1}\|_{L^2(Q_0)}  \leq \tfrac{1}{2} \| g_k \|_{L^2(Q_0)} \right\}\\
	K_4 &\coloneqq \Z_{\geq k^*} \setminus K_3.
\end{align}

\begin{lem}\label{l:conseq-second-order}
	Suppose that $\{k+1,\dots,M\} \subseteq K_3$. Then, we have 
	\begin{align}\label{e:second-derivative}
		\sum_{i =k}^M \| \cradial_{i} g_{i} \|_{L^2(Q_0)}^2 \lesssim \| \cradial_{k} g_k  \|_{L^2(Q_0)}^2. 
	\end{align}
\end{lem}

\begin{proof}
	For brevity, let us denote $\|\cdot\|_{L^2(Q_0)}$ by $\|\cdot\|.$ Squaring both sides and using the inequality $\sum_{i=k}^M a_i^2 \leq (\sum_{i=k}^M a_i)^2$ for $a_i \geq 0,$ it suffices to show 
	\begin{align}
		\sum_{i =k+1}^M \| \cradial_{i} g_{i}  \| \lesssim \| \cradial_{k} g_k \|. 
	\end{align}
	Note that we compute the sum from $i=k+1$ (the $i=k$ term in \eqref{e:second-derivative} is clearly bounded from above by the right-hand side). Now, using the triangle inequality, swapping the order of summation, summing over a geometric series and using the definition of $K_3,$ we estimate the left-hand side of \eqref{e:second-derivative} as follows, 
	\begin{align}
		\sum_{i =k+1}^M \| \cradial_{i} g_{i} \|  &\leq \sum_{i =k+1}^M \left( \frac{\cradial_{i}}{\cradial_{k}}\right)\| \cradial_{k} g_{k} \| + \sum_{i =k+1}^M\sum_{j=k}^{i-1} \left( \frac{\cradial_{i}}{\cradial_{j}}\right) \| \cradial_{j} (g_{{j+1}} - g_{j})\| \\
		&= \sum_{i =k+1}^M \left( \frac{\cradial_{i}}{\cradial_{k}}\right)\| \cradial_{k} g_{k} \| + \sum_{j=k}^{M-1} \sum_{i = j+1}^M \left( \frac{\cradial_{i}}{\cradial_{j}}\right) \| \cradial_{j} (g_{{j+1}} - g_{j})\| \\
		&\leq \| \cradial_{k} g_{k} \|  + \sum_{j=k}^{M-1} \| \cradial_{j} (g_{{j+1}} - g_{j})\| \leq \| \cradial_{k} g_{k}  \| +\frac{1}{2} \sum_{j=k}^{M-1} \| \cradial_{{j+1}} g_{{j+1}} \|.
	\end{align}
	Rearranging the above inequality gives the desired estimate. 
\end{proof}

\begin{lem}\label{l:second-deriv-M}
	For each integer $M \geq 1$, we have 
	\begin{align}\label{e:second-derivative-M}
		\sum_{k=k^*}^M \| \cradial_k g_k \|^2_{L^2(Q_0)} \lesssim \sum_{k = k^*-1}^{M-1} \| \cradial_k (g_{k+1} - g_k) \|^2_{L^2(Q_0)} .
	\end{align}
\end{lem}

\begin{proof}
	Again, for brevity, let us denote $\|\cdot\|_{L^2(Q_0)}$ by $\|\cdot\|.$ Let $\{k_i\}_{i =1}^N$ be an enumeration of the elements in $K_4$ between $k^*$ and $M$ (inclusive) ordered such that $k_i < k_{i+1}$ for all $i$. Since $g_{k^*-1} \equiv 0$, it follows that $k^* \in K_4$ (otherwise, $\|g_{k^*}\| \leq \tfrac{1}{2}\|g_{k^*}\|$ by the definition of $K_3$). Thus, the left-hand side of \eqref{e:second-derivative-M} can be written as 
	\begin{align}
		\sum_{i=1}^N \sum_{k=k_i}^{k_{i+1}-1} \| \cradial_k g_k \|^2 + \sum_{k=k_N}^M \| \cradial_k g_k \|^2. 
	\end{align}
	By Lemma \ref{l:conseq-second-order} and the definition of $K_4$, we now have 
	\begin{align}
		\begin{split}\label{e:first}
		\sum_{k=k^*}^M \| \cradial_k g_k  \|^2 \lesssim  \sum_{i=1}^N \| \cradial_{k_i} g_{k_i} \|^2 \lesssim \sum_{i=1}^N \| \cradial_{k_i-1} (g_{k_i} - g_{k_i-1}) \|^2,
		\end{split}
	\end{align}
	which implies \eqref{e:second-derivative-M}. 
\end{proof}

\begin{lem}\label{l:hess-bound}
	For each integer $M \geq 1$ we have 
	\begin{align}
		\sum_{k=k^*}^M \| \cradial_k g_k \|^2_{L^2(Q_0)} \lesssim_{\Cdoub, \CPI} \| \nabla f \|^2_{L^2(3\API B_{Q_0})}. 
	\end{align}
\end{lem}

\begin{proof}
	First, by Lemma \ref{l:second-deriv-M}, 
	\begin{align}
		\begin{split}\label{e:above}
			\sum_{k=k^*}^M \| \cradial_k g_k \|^2_{L^2(Q_0)} &\lesssim \sum_{k=k^*-1}^{M-1} \| \cradial_k (g_{k+1} - g_k) \|^2_{L^2(Q_0)} \\
			&= \sum_{k =k^*-1}^{M-1} \sum_{Q \in \calD_{k+1}(Q_0)} \| \cradial_k (g_{k+1} - g_k) \|^2_{L^2(Q)}.
		\end{split}
	\end{align}
	
	Fix some $Q \in \calD_{k+1}$ and let $Q' \in \calD_{k}$ such that $Q \subseteq Q'.$ Suppose to begin with that $k \geq k^*$. In this case, recall from \eqref{e:optimal-l} and \eqref{e:def-g_k} that $g_{k+1}\mathds{1}_Q = |\myHess(l_Q)|_\HS \mathds{1}_Q$, where $l_Q \colon X \to \R$ is harmonic on $R_Q \supseteq 10B_Q.$ Similarly, $g_{k} \mathds{1}_{Q'} = |\myHess (l_{Q'})|_\HS\mathds{1}_{Q'}$ for some function $l_{Q'} \colon X \to \R$ which is harmonic on $R_{Q'} \supseteq 10B_{Q'} \supseteq 10B_Q$. By applying, Lemma \ref{l:caccio-replacement}, Lemma \ref{l:reverse-triangle-inequality}, Corollary \ref{c:second-order-caccio} and \eqref{e:optimal-l}, we have 
	\begin{align}
		\begin{split}\label{e:third}
		\| \cradial_k (g_{k+1} - g_k) \|^2_{L^2(Q)} &= \| \cradial_k |\myHess (l_Q - l_{Q'}) |_\HS \|^2_{L^2(Q)} \\
		&\lesssim_{\cRCD(X),\Cdoub,\CPI} \left\| \frac{l_Q - l_{Q'}}{\cradial_k}\right\|^2_{L^2(4B_Q)} \\
		&\lesssim \left\| \frac{f - l_Q}{\cradial_{k+1}}\right\|_{L^2(4B_Q)}^2 + \left\| \frac{f - l_{Q'}}{\cradial_k} \right\|_{L^2(4B_Q)}^2 \\
		&\lesssim \har{f}(R_Q)^2 \mu(R_Q) + \har{f}(R_{Q'})^2 \mu(R_{Q'}). 
		\end{split}
	\end{align}
	If $k = k^*-1$, a similar argument implies 
	\begin{align}\label{e:fourth}
		\| \cradial_0 (g_{k^*} - g_{k^*-1}) \|^2_{L^2(Q)} \lesssim_{\cRCD(X),\Cdoub,\CPI} \har{f}(R_Q)^2 \mu(R_Q) + \| \nabla f \|_{L^2(3\API B_{Q_0})}^2. 
	\end{align}
	Indeed, the only modification one needs to make is the final estimate in \eqref{e:third}. Recall that for $Q' \in \cubes_{k^*-1}$, we have $l_{Q'} \equiv \langle f \rangle_{3B_{Q'}}$. It then follows from the Poincar\'e inequality and the fact that $4B_Q \subseteq 3B_{Q'}$ that  
	 \begin{align}
	 	\left\| \frac{f - l_{Q'}}{\cradial_k} \right\|_{L^2(4B_Q)}^2 \leq \left\| \frac{f - l_{Q'}}{\cradial_k} \right\|_{L^2(3B_{Q'})}^2 \lesssim_{\CPI}  \| \nabla f \|^2_{L^2(3\API B_{Q_0})}. 
	 \end{align}
	
	Returning to \eqref{e:above}, by applying \eqref{e:third}, \eqref{e:fourth}, and Lemma \ref{l:control-R_Q}, we now have 
	\begin{align}
		\sum_{k=k^*}^M \| \cradial_k g_k \|^2_{L^2(Q_0)}  &\lesssim_{\cRCD(X)} \| \nabla f \|_{L^2(3\API B_{Q_0})}^2 + \sum_{k=k^*}^{M} \sum_{Q \in \cubes_k(Q_0)} \har{f}(R_Q)^2\mu(R_Q) \\
		&\lesssim_{\cRCD(X)} \| \nabla f \|^2_{L^2(3\API B_{Q_0})},
	\end{align}
	which finishes the proof. 
\end{proof}

\section{Estimating the gradient}\label{s:if-statement}

In this section, we prove the remaining implication of Theorem \ref{t:RCD-2}. Namely, the norm of the gradient can be controlled by the square-sum of the $\har{f}$ coefficients. Recall, this also implies the remaining implication in Theorem \ref{t:RCD}. Fix a space $(X,d,\mu)$ that is $\RCD(K,N)$ space for some $K \in \R$ and $N \in (1,\infty)$, and assume that $\diam(X) < \infty$ in the case that $K < 0$. Fix also some $f \in L^2(X)$ for the remainder of the section. \\

Again, we would like to discretise via Lemma \ref{l:continuous-to-discrete}. Suppose for the moment that $\diam(X) < \infty$, and let $k^* \in \Z$ be the smallest integer such that $30\rho^{k^*} \leq \diam(X)/2$. For use later, we note that minimality implies 
\begin{align}\label{e:k^*-choice}
	(\rho/6)\diam(X) < 5\rho^{k^*}. 
\end{align}
With this choice of $k^*$, if $Q \in \cubes_{k^*}$, then $6B_Q$ is contained in a ball of radius $\diam(X)/2$. Moreover, there are at most finitely many (depending on $\Cdoub$) cubes in $\cubes_k$, which translates to the fact that there are only finitely many corresponding balls. Thus, by Lemma \ref{l:continuous-to-discrete} (2), 
\begin{align}
	\sum_{k=k^*}^\infty \sum_{Q \in \cubes_k} \har{f}(6B_Q)^2 \mu(Q) \lesssim_{\Cdoub} \int_0^{\diam(X)} \int_X \har{f}(x,r)^2 \, \frac{d\mu dr}{r}. 
\end{align}
If $\diam(X) = \infty$, the same estimate follows immediately from Lemma \ref{l:continuous-to-discrete} (2) with $k^* = -\infty$. Hence, it suffices to prove the following.

\begin{prop}\label{p:gradient-bound}
	Let $\cubes$ be a system of Christ-David cubes on $X$ and let $k^*$ be as above. If 
	\begin{align}\label{e:sufficient-gradient-bound}
		\sum_{k=k^*}^\infty \sum_{Q \in \cubes_k} \har{f}(6B_Q)^2 \mu(Q) < \infty, 
	\end{align}
	then $f \in N^{1,2}(X).$ In this case, 
	\begin{align}\label{e:gradient-bound}
		\| \nabla f  \|_{L^2(X)}^2 \leq C_1 \left\| \frac{f - \langle f \rangle_X}{\diam(X)} \right\|^2_{L^2(X)} + C_2 \sum_{k=k^*}^\infty \sum_{Q \in \cubes_k}\har{f}(6B_Q)^2 \mu(Q)
	\end{align}
	for some absolute constant $C_1$ and some constant $C_2$ depending only on $\cRCD(X)$, as in  \eqref{e:cRCD(X)}. 
\end{prop}

To prove Proposition \ref{p:gradient-bound}, we will make use of various properties of the heat semigroup. Consider the operator $-\Delta$ which is a non-negative definite self-adjoint operator on $X$. The \textit{heat semigroup} associated to $-\Delta$ is the family $\{H_t\}_{t \geq 0}$ of operators from $L^2(X)$ into itself given by $H_t \coloneqq e^{\Delta t}$. It satisfies the following list of properties (see, for example, \cite{fukushima2011dirichlet}).

\begin{enumerate}[label=(H\arabic*),ref=H\arabic*]
	\item \label{H:1} For each $s,t > 0$, 
	\begin{align}
		H_{s+t} = H_s H_t. 
	\end{align}
	\item \label{H:2} For each $g \in L^2(X)$,  
	\begin{align}
		\|H_t g - g \|_{L^2(X)} \to 0 \mbox{ as } t \to 0. 
	\end{align}
	\item \label{H:3} For each $t > 0$, if $\| \cdot \|_{2 \to 2}$ denotes the operator norm from $L^2$ to $L^2$, then 
	\begin{align}
		\| H_t \|_{2 \to 2} \leq 1 \quad \mbox{ and } \quad \| \nabla H_t \|_{2 \to 2} \leq ct^{-\frac{1}{2}}.
	\end{align}
	\item \label{H:4} For each $t > 0$ and $g \in L^2(X)$, 
	\begin{align}
		H_t g \in D(\Delta) \quad \mbox{ and } \quad \partial_t H_t g  = \Delta  H_t g. 
	\end{align}
	\item \label{H:5} For any continuous $\phi \colon [0,\infty) \to [0,\infty)$, $\phi(-\Delta)$ is non-negative definite self-adjoint operator which commutes with $H_t$, $t > 0$.
\end{enumerate}

\

It will be useful to note that, by \eqref{e:harmonics} and the fact that $(-\Delta)^\frac{1}{2}$ is non-negative self-adjoint, we have
\begin{align}\label{e:rep-A}
	\| \nabla g \|^2 = \langle -\Delta g , g \rangle  = \| (-\Delta)^\frac{1}{2} g \|^2  
\end{align} 
for all $g \in D(\Delta)$. \\

Finally, we will need the following lemma, which is essentially contained in \cite{hytonen2019heat} (see Equations (55)--(58) there). We include the short proof for completeness.

\begin{lem}\label{l:heat-telescope}
	For every $g \in L^2(X)$ we have 
	\begin{align}
		\int_0^\infty \| \sqrt{t} \nabla  H_t g \|_{L^2(X)}^2 \, \frac{dt}{t} \lesssim \int_0^\infty \| \sqrt{t} \nabla H_t(H_{3t} - H_{t})g \|_{L^2(X)}^2 \, \frac{dt}{t}. 
	\end{align}
\end{lem}

\begin{proof}
 For brevity, let us write $\|\cdot\| \coloneqq \|\cdot\|_{L^2(X)}$. By \eqref{H:3} and \eqref{e:rep-A}, we know that $\| (-\Delta)^\frac{1}{2} H_t g \| \to 0$ as $t \to \infty$. Thus, we can write 

	\begin{align}
		(-\Delta)^\frac{1}{2}H_t g &= \sum_{k=-1}^\infty (-\Delta)^\frac{1}{2}\left( H_{2^{k+1}t}g  - H_{2^{k+2}t}g \right) \\
		&= \sum_{k=-1}^\infty (-\Delta)^\frac{1}{2} H_{2^kt}\left(H_{ 2^kt} - H_{3\cdot 2^kt} \right) g .
	\end{align}
	Applying this, equation \eqref{e:rep-A}, the triangle inequality (in $L^2( (0,\infty), dt/t ; L^2(X,\R) )$), and making a change of variables $s = 2^\frac{k}{2}$, we have 
	\begin{align}
		\left (	\int_0^\infty \| \sqrt{t} \nabla  H_t g \|^2 \, \frac{dt}{t} \right)^\frac{1}{2} &\leq \sum_{k=-1}^\infty \left( 	\int_0^\infty \| \sqrt{t} \nabla   H_{2^kt}\left(H_{3\cdot 2^kt} - H_{2^kt} \right) g \|^2 \, \frac{dt}{t} \right)^\frac{1}{2} \\
		&= \sum_{k=-1}^\infty 2^{-\frac{k}{2}} \left(  \int_0^\infty \| \sqrt{t} \nabla   H_{t}\left(H_{3t} - H_{t} \right) g \|^2 \, \frac{dt}{t} \right)^\frac{1}{2} \\
		&\lesssim \left(  \int_0^\infty \| \sqrt{t} \nabla   H_{t}\left(H_{3t} - H_{t} \right) g \|^2 \, \frac{dt}{t} \right)^\frac{1}{2},
	\end{align}
	as required. 
\end{proof}

With these properties established, let us return to the proof of Proposition \ref{p:gradient-bound}. First, \eqref{H:3} implies that $H_sf \in N^{1,2}(X)$ for all $s > 0$. Then, since the mapping $g \mapsto \| \nabla g \|_{L^2(X)}$ is lower semicontinuous with respect to convergence in $L^2$ (see \cite[Theorem 2.5]{cheeger1999differentiability}), and $H_s f \to f$ in $L^2$ as $s \to 0$ (recall \eqref{H:2}), it suffices to show the following. 

\begin{lem}\label{l:sufficient}
	Let $\cubes$ be a system of Christ-David cubes on $X$ and let $k^*$ be as above. We have 
	\begin{align}\label{e:discrete-C}
	\| \nabla  H_s f \|_{L^2(X)}^2\leq C_1 \left\| \frac{f - \langle f \rangle_X}{\diam(X)} \right\|^2_{L^2(X)} + C_2 \sum_{k=k^*}^\infty \sum_{Q \in \cubes_k}\har{f}(6B_Q)^2 \mu(Q)
	\end{align}
	for some absolute constant $C_1$ and some constant $C_2$ depending only on $\cRCD(X)$ i.e., both are independent of $s$.
\end{lem}

The proof of Lemma \ref{l:sufficient} will be carried out below. Let us fix a system of Christ-David cubes $\cubes$ and let us fix some $s > 0$. We start by showing the following. 

\begin{lem}\label{l:6-1}
	Let $\alpha \coloneqq (\rho / 6)$. We have 
	\begin{align}\label{e:htf-t}
		\| \nabla  H_s f \|_{L^2(X)}^2 \lesssim \left\| \frac{f - \langle f \rangle_X}{\diam(X)} \right\|^2_{L^2(X)} +  \int_0^{\alpha \diam(X)} \left\| \frac{H_{2t_2}f - f}{t} \right\|_{L^2(X)}^2  \, \frac{dt}{t}.
	\end{align}
	 
\end{lem}

\begin{proof}
	For brevity, let us denote $f_s \coloneqq H_s f$. Let $\{E_\lambda \colon \lambda > 0 \}$ be the unique spectral family of projection operators on $L^2(X)$ associated to $-\Delta$, and consider the spectral representation $-\Delta = \int_0^\infty \lambda \, dE_{\lambda}.$ Since $f_s \in D(\Delta)$, we have by \eqref{e:rep-A}, standard identities in spectral theory, and the identity $\lambda = 2\int_0^\infty \lambda^2 e^{- 2\lambda t} \, dt$, that  
	\begin{align}
		\| \nabla f_s \|^2 = \int_0^\infty \lambda \,  d\langle E_\lambda f_s,f_s\rangle &= 2\int_0^\infty \int_0^\infty  \lambda^2 e^{- 2\lambda t} \, d\langle E_\lambda f_s,f_s\rangle   dt \\
		&=2 \int_0^\infty \| \nabla H_t g_s  \|^2 \, dt,
	\end{align}
	where $g_s \coloneqq (-\Delta)^\frac{1}{2} f_s$ and $\int G(x) dF(x)$ is the Lebesgue–Stieltjes integral. Multiplying and dividing by $t$, and applying Lemma \ref{l:heat-telescope} and \eqref{e:rep-A}, we get 
	\begin{align}
		\| \nabla f_s \|^2 = \int_0^\infty \| \sqrt{t} \nabla  H_t g_s\|^2 \, \frac{dt}{t} \lesssim \int_0^\infty \| \sqrt{t} (-\Delta)^\frac{1}{2} H_t \left(H_{3t} - H_t \right)g_s \|^2 \, \frac{dt}{t}.  
	\end{align}
	Applying the first estimate in \eqref{H:3} (on three separate occasions), commuting $(-\Delta)^\frac{1}{2}$ with $H_t,$ applying \eqref{e:rep-A} and making a change of variable $t = s^2,$ this gives
	\begin{align}
		\begin{split}\label{e:almost}
		\|\nabla f_s\|^2 &\lesssim \int_0^\infty \| H_{3t} g_s - H_t g_s \|^2 \, \frac{dt}{t} = \int_0^\infty \| \nabla H_t (H_{2t}f_s - f_s ) \|^2 \frac{dt}{t} \\
		&\lesssim \int_0^\infty \left\| \frac{H_{2t^2}f_s - f_s}{t} \right\|^2 \, \frac{dt}{t} \leq \int_0^\infty \left\| \frac{H_{2t^2}f - f}{t} \right\|^2  \, \frac{dt}{t}.
		\end{split}
	\end{align}
	
	If $\diam(X) = \infty$, then the above inequality is exactly \eqref{e:htf-t} and we are done. Suppose instead that $\diam(X) < \infty$. By doubling, we know $\mu(X) < \infty$. Thus, $\langle f \rangle_X \in L^2(X)$. Since constant functions are invariant under $H_t$ for each $t > 0$, we can estimate using \eqref{H:3} to get 
	\begin{align}
		\|H_{2t^2}f - f\| \leq \| H_{2t^2}f - H_{2t^2}\langle f \rangle_X\| + \| f- \langle f \rangle_X\| \leq 2\|f - \langle f \rangle_X\|. 
	\end{align}
	Plugging this into \eqref{e:almost}, the final term is bounded from above by a constant multiple of 
	\begin{align}
		\|f - \langle f \rangle_X\|^2 \int_{\alpha \diam(X)}^\infty \frac{dt}{t^3} + \int_0^{\alpha \diam(X)} \left\| \frac{H_{2t^2}f - f}{t} \right\|^2,
	\end{align}
	which is comparable to the right-hand side of \eqref{e:htf-t}. 
\end{proof}

The first term on the right-hand side of \eqref{e:htf-t} is exactly the first term appearing on the right-hand side of \eqref{e:discrete-C}. Thus, our goal for the remainder of the section is to estimate the second term on right-hand side of \eqref{e:htf-t}. We begin by breaking the integral into two pieces. 

Define, for each $k \geq k^*,$ a function $p_k \colon X \to \R$ as follows. First, by applying Proposition \ref{l:cut-off}, construct a partition of unity $\{\theta_Q\}_{Q \in \cubes_k}$ satisfying the properties
\begin{align}
	\supp(\theta_Q) \subseteq 3B_Q, \ | \nabla  \theta_Q(y) | \lesssim \ell(Q)^{-1}, \ |\Delta \theta_Q(y) | \lesssim \ell(Q)^{-2}, \mbox{ and } \sum_{Q \in \cubes_k} \theta_Q(y) = 1
\end{align}
for all $Q \in \cubes_k$ and $y \in X$, where the implicit constants depend on $\cRCD(X)$. For $Q \in \cubes_k,$ let $h_Q \colon X \to \R$ be a function which is harmonic on $6B_Q$ and satisfies 
\begin{align}
	\har{f}(6B_Q;h_Q) \lesssim \har{f}(6B_Q) . 
\end{align}Then, define $p_k$ by 
\begin{align}
	p_k \coloneqq \sum_{Q \in \cubes_k} \theta_Q h_Q.
\end{align}

Let $\cradial_k$ be as in \eqref{e:crad}, and recall that this is chosen so that $\ell(Q) = \cradial_k$ for each $Q \in \cubes_k$. By \eqref{e:k^*-choice} and our choice of $\alpha$ from Lemma \ref{l:6-1}, we know that $\alpha \diam(X) \leq \cradial_{k^*}$. This and \eqref{H:3} imply 
\begin{align}
	 &\int_0^{\alpha \diam(X)} \left\| \frac{H_{2t_2}f - f}{t} \right\|_{L^2(X)}^2  \, \frac{dt}{t}  \\
	 &\hspace{4em}\lesssim \sum_{k=k^*}^\infty \int_{\cradial_{k+1}}^{\cradial_k}  \left\| \frac{f - p_k}{t} \right\|^2 \,\frac{dt}{t} + \sum_{k=k^*}^\infty \int_{\cradial_{k+1}}^{\cradial_k}  \left\| \frac{H_{2t^2}p_k - p_k}{t} \right\|^2 \, \frac{dt}{t} \eqqcolon I_1 + I_2. 
\end{align}

\begin{lem}
	We have 
	\begin{align}
		I_1 \lesssim_{\Cdoub}  \sum_{k=k^*}^\infty \sum_{Q \in \cubes_k} \har{f}(6B_Q)^2 \mu(Q). 
	\end{align}
\end{lem}

\begin{proof}
	Fix some $k \geq k^*$ and $x \in Q \in \cubes_k$ for the moment. Since $\supp(\theta_R) \subseteq 3B_R$ and $\{\theta_R\}_{R \in \cubes_k}$ is a partition of unity, we have 
	\begin{align}
		f(x) - p_k(x) =  \sum_{\substack{R \in \cubes_k \\ Q \cap 3B_R \neq\emptyset}} \theta_R(x) (f(x) - h_R(x)) .
	\end{align}
	Thus, by applying Lemma \ref{l:bounded-overlap} and the inequality 
	\begin{align}\label{e:square-triangle}
		\left(\sum_{i=1}^N a_i\right)^2 \leq N^2 \sum_{i=1}^N a_i^2, \quad N \in \N, 
	\end{align} 
	we get 
	\begin{align}\label{e:earlier}
		|f(x) - p_k(x)|^2 \lesssim_{\Cdoub}  \sum_{\substack{R \in \cubes_k \\ Q\cap 3B_R \neq \emptyset}} | f(x) - h_R(x) |^2.
	\end{align}
	In particular, recalling the definition of $\har{f},$ for every $t \in (\cradial_{k+1},\cradial_k)$ we have
	\begin{align}
		\int_Q \left( \frac{f - p_k}{t} \right)^2 \,d\mu &\lesssim_{\Cdoub}	 \sum_{\substack{R \in \cubes_k \\ Q \cap 3B_R \neq\emptyset }} \har{f}(3B_R;h_R)^2 \mu(R).
	\end{align}
	Using this estimate, along with Lemma \ref{l:bounded-overlap} and our choice of $h_Q$ gives 
	\begin{align}
		I_1 \lesssim_{\Cdoub} \sum_{k=k^*}^\infty \sum_{Q \in \cubes_k} \sum_{\substack{R \in \cubes_k \\ Q \cap 3B_R \neq\emptyset}} \har{f}(3B_R;h_Q)^2 \mu(R) \lesssim_{\Cdoub} \sum_{k=k^*}^\infty \sum_{Q \in \cubes_k} \har{f}(6B_Q)^2 \mu(Q).
	\end{align}
\end{proof}

The next lemma finishes the proof of Lemma \ref{l:sufficient}. 
	
	\begin{lem}
		We have 
		\begin{align}
			I_2 \lesssim_{\cRCD(X)}  \sum_{k=k^*}^\infty \sum_{Q \in \cubes_k} \har{f}(6B_Q)^2 \mu(Q). 
		\end{align}
	\end{lem}
	
	\begin{proof}
Recall from \eqref{H:4} that $\partial_u H_u = \Delta H_u$ for each $u > 0$. Thus, we can write 
	\begin{align}
		I_2 = \sum_{k=k^*}^\infty \int_{\cradial_{k+1}}^{\cradial_k}  \left\| \frac{1}{t} \int_0^{2t^2} \partial_u H_u p_k \, du  \right\|^2 \, \frac{dt}{t} = \sum_{k=k^*}^\infty \int_{\cradial_{k+1}}^{\cradial_k}  \left\| \frac{1}{t} \int_0^{2t^2} \Delta H_u p_k \, du  \right\|^2 \, \frac{dt}{t}.
	\end{align}
	By H\"older's inequality, using that $\Delta$ and $H_u$ commute, and $\| H_u \|_{2 \to 2} \leq 1,$ we find that for any $k \geq k^*$ and $t \in (\cradial_{k+1},\cradial_k),$
	\begin{align}
		\left\| \frac{1}{t} \int_0^{2t^2} \Delta H_u p_k \, du  \right\|^2 &\lesssim \int_X t^2 \dashint_0^{2t^2} | H_u \Delta p_k |^2 \, dud\mu  \lesssim \dashint_0^{2t^2} \| \cradial_k H_u \Delta p_k \|^2 \, du \\
		&\leq \| \cradial_k \Delta p_k\|^2 . 
	\end{align}
	
	We will show in the remainder of the proof that  
	\begin{align}
		\| \cradial_k \Delta p_k \|^2 = \sum_{Q \in \cubes_k} \int_Q | \cradial_k \Delta p_k |^2 \, d\mu \lesssim_{\cRCD(X)}  \sum_{Q \in \cubes_k} \har{f}(6B_Q)^2 \mu(Q). 
	\end{align}
	Fix some $Q \in \cubes_k$ for the moment. Since $\{\theta_R\}_{R \in \cubes_k}$ is a partition of unity, it follows that 
	\begin{align}
		h_Q - p_k = \sum_{R \in \cubes_k} \theta_R \left( h_Q - h_R \right).
	\end{align}
	In particular, applying the chain rule for $\Delta$ and recalling that the function $h_Q$ and $h_R$ are harmonic on $6B_Q$ and $6B_R,$ respectively, if $x \in Q$ then 
	\begin{align}
		\Delta p_k(x) &= \Delta (h_Q - p_k)(x)  \\
		&= \sum_{R \in \cubes_k} \left[  \Delta \theta_R(x)(h_Q - h_R)(x) + \nabla \theta_R(x) \cdot \nabla (h_Q - h_R)(x)\right].
	\end{align}
	Using the bounds on the gradient and Laplacian of the cut-off functions, and estimating as we did in \eqref{e:earlier} (by applying Lemma \ref{l:bounded-overlap} and \eqref{e:square-triangle}), we have 
	\begin{align}\label{e:second-to-last}
		| \Delta p_k |^2 \lesssim_{\cRCD(X), \Cdoub} \sum_{\substack{R \in \cubes_k \\ Q \cap 3B_R \neq\emptyset}} \left| \frac{h_Q-h_R}{\cradial_k^2} \right|^2 + \left| \frac{\nabla h_Q - \nabla h_R}{\cradial_k}\right|^2 \quad \mbox{ on } Q. 
	\end{align}
	Since $h_Q$ and $h_R$ are harmonic on $6B_Q$ and $6B_R$, respectively, and $Q \cap 3B_R \subseteq 3B_Q \cap 3B_R,$ we may apply Caccioppoli's inequality (Lemma \ref{l:caccio-replacement}) to obtain
	\begin{align}
		\begin{split}\label{e:last}
		\int_Q \sum_{\substack{R \in \cubes_k \\ Q \cap 3B_R \neq\emptyset}} | \nabla h_Q - \nabla h_R |^2 \, d\mu &\lesssim_{\cRCD(X)} \sum_{\substack{R \in \cubes_k \\ Q \cap 3B_R \neq\emptyset}} \int_{6B_Q \cap 6B_R} \left( \frac{h_Q - h_R}{\cradial_k}\right)^2 \, d\mu \\
		&\lesssim_{\Cdoub} \sum_{\substack{R \in \cubes_k \\ Q \cap 3B_R \neq\emptyset}} \int_{6B_R} \left( \frac{f - h_R}{\ell(R)} \right)^2 \, d\mu  \\
		&\lesssim_{\Cdoub} \sum_{\substack{R \in \cubes_k \\ Q \cap 3B_R \neq\emptyset}} \har{f}(6B_R)^2 \mu(R),
		\end{split}
	\end{align}
	where in the second inequality we added and subtracted $f$, and applied the triangle inequality. Combining \eqref{e:second-to-last} and \eqref{e:last} with Lemma \ref{l:bounded-overlap} and using our choice of $h_Q$, we obtain 
	\begin{align}
		\| \cradial_k \Delta p_k \|^2 &= \sum_{Q \in \cubes_k} \int_Q | \cradial_k \Delta p_k|^2 \, d\mu \lesssim_{\cRCD(X)}  \sum_{Q \in \cubes_k} \sum_{\substack{R \in \cubes_k \\ Q \cap 3B_R \neq\emptyset }} \har{f}(6B_R)^2 \mu(R) \\
		&\lesssim_{\Cdoub} \sum_{Q \in \cubes_k} \har{f}(6B_Q)^2 \mu(Q),
	\end{align}
	as required. 
\end{proof}

\appendix

\section{Proof of Lemma \ref{l:continuous-to-discrete}}

For each $k \in \Z$, we will denote $\cradial_k \coloneqq 5\rho^k$. 

	\begin{proof}[Proof of implication (1)] Fix $z,R,k$ and $\lambda$ as in the statement. Let $\mathcal{Q}_m$ denote the collection of $Q \in \cubes_m$ such that $Q \cap B(x,R) \neq\emptyset$. Since $R \leq \cradial_k$, we first observe that
	\begin{align}\label{e:lala}
		H_D(f,z,R) \lesssim  \sum_{m=k}^\infty \sum_{Q \in \calQ_m} \dashint_{\cradial_{m+1}}^{\cradial_m} \int_Q \har{f,D}(x,r)^2 \, d\mu dr . 
	\end{align}
	
	Fix some $m \geq k$ and $Q \in \calQ_m$ for the moment, and consider $(x,r) \in Q \times (\cradial_{m+1},\cradial_m)$. Since $x \in Q$, $r \leq \cradial_m$ and $\lambda \geq 3$, it follows from the triangle inequality that $B(x,r) \subseteq 3B_Q \subseteq \lambda B_Q$. Using this, and the fact that $\lambda \ell(Q) = \lambda \cradial_m \lesssim_{\lambda} r$, \eqref{e:semi-cont-RCD} now implies 
	\begin{align}
		\har{f,D}(B(x,r)) \lesssim_{\lambda , \Cdoub} \har{f,D}(\lambda B_Q) .
	\end{align}
	Since the above estimate holds for arbitrary $(x,r) \in Q \times (\cradial_{m+1},\cradial_m)$, plugging it back into \eqref{e:lala} gives 
	\begin{align}
		H_D(f,z,R) &\lesssim_{\lambda , \Cdoub} \sum_{m=k}^\infty \sum_{Q \in \calQ_m} \dashint_{\cradial_{m+1}}^{\cradial_m} \int_Q \har{f,D}(\lambda B_Q) ^2 \, d\mu dr \\
		&=\sum_{m=k}^\infty \sum_{Q \in \calQ_m} \har{f,D}(\lambda B_Q) ^2 \mu(Q). 
	\end{align}
	Since, for each $Q \in \calQ_m$ with $m \geq k$ there exists $Q' \in \calQ_k$ such that $Q \subseteq Q'$, this completes the proof of the first implication. 
	\end{proof}
	
	\begin{proof}[Proof of implication (2)] We begin by showing how the first part of (2) implies the second part. Indeed, if the first part holds, for each $k \in \Z$ we have 
	\begin{align}
		\sum_{Q \in \cubes_k} H_D(f,\cubes(Q),\lambda) \lesssim_{\Cdoub} \sum_{Q \in \cubes_k} H_D(f,x_Q,\lambda' \ell(Q)), 
	\end{align}
	where $\lambda' \coloneqq 2 \lambda \rho^{-1}$. By Lemma \ref{l:bounded-overlap}, the balls $\{\lambda' B_Q\}_{Q \in \cubes_k}$ have bounded overlap depending on $\lambda$ and $\Cdoub$. Thus, 
	\begin{align}
		\sum_{Q \in \cubes_k} H_D(f,\cubes(Q),\lambda) \lesssim_{\lambda,\Cdoub} \int_0^{\lambda' \cradial_k} \int_X \har{f,D}(x,r)^2 \, \frac{d\mu dr}{r} \leq H_D(f,x,\infty). 
	\end{align}
	Taking $k \to - \infty$ finishes the proof of the second part. 
	
	Now, let us prove the first part of (2). Fix some $k,Q$ and $\lambda$ as in the statement. Unravelling the definition of $H_D(f,\cubes(Q),\lambda)$, we have 
	\begin{align}\label{e:lala-2}
		 H_D(f,\cubes(Q),\lambda)  &= \sum_{m = k}^\infty \sum_{T \in \cubes_m(Q)}\dashint_{2\lambda \cradial_{m}}^{2\lambda \cradial_{m-1}} \int_T \har{f,D}(\lambda B_T)^2 \, d\mu dr 
		\end{align}
		Fix some $m \geq k$ and $T \in \cubes_m(Q)$ for the moment. If $(x,r) \in T \times (2\lambda\cradial_m , 2\lambda \cradial_{m-1})$, then it follows from the triangle inequality that $\lambda B_T \subseteq B(x,r)$. Since also $r \sim \lambda \ell(Q)$, it follows from \eqref{e:semi-cont-RCD} that 
		\begin{align}
			 \har{f,D}(\lambda B_T) \lesssim_{\Cdoub} \har{f,D}(x,r). 
		\end{align}
		Since the above estimate holds for arbitrary $(x,r) \in T \times (2\lambda\cradial_m , 2\lambda \cradial_{m-1})$, plugging it back in \eqref{e:lala-2} gives 
		\begin{align}
			H_D(f,\cubes(Q),\lambda) &\lesssim_{\Cdoub} \sum_{m = k}^\infty \sum_{T \in \cubes_m(Q)}\int_{2\lambda \cradial_{m}}^{2\lambda \cradial_{m-1}} \int_T \har{f,D}(x,r)^2 \, \frac{d\mu dr}{r} \\
			&= \int_{0}^{2\lambda \cradial_{k-1}} \int_Q \har{f,D}(x,r)^2 \, \frac{d\mu dr}{r}  \leq H_D(f,x_Q,2\lambda\rho^{-1}\ell(Q)). 
		\end{align}
		\end{proof}

\bibliography{Ref.bib}
\bibliographystyle{alpha-1.bst}
\end{document}